\documentclass[12pt]{amsart}
\usepackage{amssymb}
\usepackage{amsmath}
\usepackage{amsfonts}
\usepackage{amsthm}
\usepackage{graphicx}
\usepackage{color}
\usepackage{setspace}
\usepackage{harvard}
\usepackage[foot]{amsaddr}

\numberwithin{equation}{section}
\theoremstyle{plain}
\newtheorem{thm}{Theorem}
\newtheorem{cor}{Corollary}

\newtheorem{LemmaQR}{Lemma QR\hspace{-1ex}}

\newtheorem{LemmaMIQ}{Lemma MIQ\hspace{-0.5ex}}

\newtheorem{AssumptionMON}{Assumption MON\hspace{-0.5ex}}
\newtheorem{AssumptionIQM}{Assumption IQM\hspace{-0.5ex}}
\newtheorem{AssumptionQR}{Assumption QR\hspace{-0.5ex}}
\theoremstyle{definition}
\newtheorem{defn}{Definition}
\theoremstyle{remark}
\newtheorem{rem}{Remark}

\begin{document}
\title[Uniform Asymptotics ]{Uniform Asymptotics for Nonparametric Quantile
Regression with an Application to Testing Monotonicity}
\thanks{Lee's work was supported by the European Research Council
(ERC-2009-StG-240910-ROMETA) and by the National Research  Foundation of Korea Grant funded by the Korean 
Government (NRF-2012S1A5A8023573).
Song acknowledges the financial support of
Social Sciences and Humanities Research Council of Canada. Whang's work was
supported by the SNU Creative Leading Researcher Grant.}
\date{\today}
\author[Lee]{Sokbae Lee$^1$ $^2$}
\address{$^1$Department of Economics, Seoul National University, 1
Gwanak-ro, Gwanak-gu, Seoul, 151-742, Republic of Korea.}
\address{$^2$Centre for Microdata Methods and Practice, Institute for Fiscal
Studies, 7 Ridgmount Street, London, WC1E 7AE, UK.}
\email{sokbae@snu.ac.kr}
\author[Song]{Kyungchul Song$^3$}
\address{$^3$Vancouver School of Economics, University of British Columbia,
997 - 1873 East Mall, Vancouver, BC, V6T 1Z1, Canada}
\email{kysong@mail.ubc.ca}
\author[Whang]{Yoon-Jae Whang$^1$}
\email{whang@snu.ac.kr}

\begin{abstract}
In this paper, we establish a uniform error rate of a Bahadur representation
for local polynomial estimators of quantile regression functions. The error
rate is uniform over a range of quantiles, a range of evaluation points in
the regressors, and over a wide class of probabilities for observed random
variables. Most of the existing results on Bahadur representations for local
polynomial quantile regression estimators apply to the fixed data generating
process. In the context of testing monotonicity  where the null
hypothesis is of a complex composite hypothesis, it is particularly relevant
to establish Bahadur expansions that hold uniformly over a large class of
data generating processes. In addition, we establish the same error rate for
bootstrap local polynomial estimators which can be useful for various
bootstrap inference. As an illustration,  we apply to testing monotonicity of quantile regression and
present Monte Carlo experiments based on this example. \bigskip

 \bigskip

%{\footnotesize \noindent \textsc{JEL Subject Classification.} C12, C14. }
\end{abstract}

\maketitle

\onehalfspacing

%\textbf{[We need to update the arXiv version of \citeasnoun{LSW} to reflect
%that we have now split the paper into two parts. -- SL]}

\clearpage

\section{Introduction}

In this paper, we establish a Bahadur representation of a local polynomial
estimator of a nonparametric quantile regression function that is uniform
over a range of quantiles, a range of evaluation points in the regressors,
and a wide class of probabilities underlying the distributions of observed
random variables. We also establish a Bahadur representation for a bootstrap
estimator of a nonparametric quantile regression function.

There are several existing results of Bahadur representation of the local
polynomial quantile regression estimator in the literature. %
\citeasnoun{Chaudhuri:91a} is the classical result on a local polynomial
quantile regression estimator with a uniform kernel. His result is pointwise
in that the representation holds at one quantile, for a fixed point, and for
a given data generating process. For recent contributions that are closely
related to this paper, see \citeasnoun{Kong/Linton/Xia:10}, %
\citeasnoun{Guerre/Sabbah:12}, \citeasnoun{Kong/Linton/Xia:13}, and %
\citeasnoun{Qu:Yoon:15}, among others. \citeasnoun{Kong/Linton/Xia:10}
obtain a Bahadur representation for a local polynomial M-estimator,
including the quantile regression estimator as a special case, for strongly
mixing stationary processes. Their result holds uniformly for a range of
evaluation points in the regressors but at a fixed quantile for a given data
generating process. \citeasnoun{Guerre/Sabbah:12} obtain Bahadur
representations that hold uniformly over a range of quantiles, a range of
evaluation points in the regressors, and a range of bandwidths for
independent and identically distributed (i.i.d.) data. However, their result
is for a fixed data generating process. \citeasnoun{Kong/Linton/Xia:13}
extend to the case when the dependent variable is randomly censored for
i.i.d. data and obtain the representation that is uniform over the
evaluation points. \citeasnoun{Qu:Yoon:15} consider estimating the
conditional quantile process nonparametrically for the i.i.d. data using
local linear regression, with a focus on quantile monotonicity. They develop
a Bahadur representation that is uniform in the quantiles but at a fixed
evaluation point for a given data generating process. It seems that our work
is the first that obtains a Bahadur representation that holds uniformly over
the triple: the quantile, the evaluation point, and the underlying
distribution. However, our result is for a fixed bandwidth, unlike %
\citeasnoun{Guerre/Sabbah:12}.

The most distinctive feature of our result is that the Bahadur
representation is uniform over a wide class of probabilities. Uniformity of
asymptotic approximation in probabilities has long drawn interest in
statistical decision theory and empirical process theory. Uniformity in
asymptotic approximation is generally crucial for procuring finite sample
stability of size or coverage probability in inference. See %
\citeasnoun{Andrews/Cheng/Guggenberger:11} for its emphasis and general
tools for uniform asymptotic results. In the recent literature of
econometrics, identifying the class of probabilities for which the
uniformity holds, and their plausibility in practice, have received growing
attention, along with increasing interests in models based on inequality
restrictions. See, for example, \citeasnoun{Andrews/Soares:10} and %
\citeasnoun{Andrews/Shi:13a} among many others.

To see the issue of uniformity, consider a simple testing problem: 
\begin{align}  \label{simple-example}
H_{0}:\forall x\in \mathcal{X},\ \frac{\partial \mathrm{Med}\left[ Y|X=x%
\right] }{\partial x} \geq 0\;vs.\;H_{1}:\exists x\in \mathcal{X},\ \frac{%
\partial \mathrm{Med}\left[ Y|X=x\right] }{\partial x} <0,
\end{align}
where $\mathrm{Med}\left[ Y|X=x\right]$ is the conditional median of $Y$
given $X=x$ and $\mathcal{X}$ is a region of interest. Then one may develop
a nonparametric test statistic using the local polynomial quantile
regression estimator (e.g. the $L_p$ statistic as in Section \ref%
{test-mono-qr}). The behavior of this nonparametric test statistic depends
crucially on the contact set $B := \{x \in \mathcal{X}: \ {\partial \mathrm{%
Med}\left[ Y|X=x\right] }/{\partial x} = 0 \}$. To emphasize the issue of
uniformity, consider a sequence of data generating processes indexed by $n$.
For example, we take the sequence of the true conditional median functions
to be $\mathrm{Med}_n\left[ Y|X=x\right] = x^3/n$ on $\mathcal{X} = [-1,1]$.
Then for each $n$, the corresponding contact set is a singleton set, that is 
$B_n = \{ 0 \}$; however, $\mathrm{Med}_n\left[ Y|X=x\right] = x^3/n$
converges to $0$ uniformly in $x \in \mathcal{X}$ as $n \rightarrow \infty$.
In other words, as $n$ gets large, the true function looks flat on $\mathcal{%
X}$, but the population contact set is always the singleton set at zero for each $n$.
This suggests that the pointwise asymptotic theory may not be adequate for
finite sample approximation. Therefore, it is important to develop uniform
asymptotics for the local polynomial quantile regression estimator by
establishing the Bahadur representation that is uniform over a large class
of probabilities.

We illustrate the usefulness of  our Bahadur representation by applying it to testing monotonicity of quantile regression that includes \eqref{simple-example} as a special case. Our proposed test uses 
  the  framework of \citeasnoun[LSW hereafter]{LSW}. They provide a general method of testing
inequality restrictions for nonparametric functions, and make use of this
paper's result in establishing sufficient conditions for one of their
results.

The remainder of the paper is as follows. Section \ref{sec:main-results}
presents the main results of the paper, Section \ref{test-mono-qr} gives an application of our main results in the context of 
testing monotonicity, Section %
\ref{sec:conclusion} concludes, and Section \ref{sec:proofs} gives all the
proofs.

\section{Uniform Asymptotics}

\label{sec:main-results}

This section provides uniform Bahadur representations 
for local polynomial quantile regression estimators 
and considers their bootstrap version as well.

\subsection{Uniform Bahadur Representation for Local Polynomial Quantile
Regression Estimators}

In this subsection, we present a Bahadur representation of a local polynomial
quantile regression estimator that can be useful for a variety of purposes.

Let $(B^{\top },X^{\top },L)^{\top }$, with $B\equiv (B_{1},...,B_{\bar{L}%
})^{\top }\in \mathbf{R}^{\bar{L}}$, and $X\in \mathbf{R}^{d}$, be a random
vector such that the joint distribution of $(B^{\top },X^{\top })^{\top }$
is absolutely continuous with respect to Lebesgue measure and $L$ is a
discrete random variable taking values from $\mathbb{N}_{L}\equiv \{1,2,...,%
\bar{L}\}$. 
For
each $x\in \mathbf{R}^{d}$ and $k\in \mathbb{N}_{L}$, we assume that the conditional
distribution of $B_{l}$ given $(X,L)=(x,k)$ is the same across $l=1,...,k$.
It is unconventional to consider a vector $B$, but it is useful to do so
here to cover the case where we observe multiple outcomes from the same conditional distribution.

Let $q_{k}(\tau |x)$ denote the $\tau $-th quantile of $B_{l}$ conditional
on $X=x$ and $L=k,$ where $\tau \in (0,1)$. That is, $P\{B_{l}\leq
q_{k}(\tau |x)|X=x,L=k\}=\tau $ for all $x$ in the support of $X$ and all $%
k\in \{1,...,\bar{L}\}$. We write%
\begin{equation*}
B_{l}=q_{k}(\tau |X)+\varepsilon _{\tau ,lk},\text{ }\tau \in (0,1),\text{
for all }k\in \{1,...,\bar{L}\},
\end{equation*}%
where $\varepsilon _{\tau ,lk}$ is a continuous random variable such that
the $\tau $-th conditional quantile of $\varepsilon _{\tau ,lk}$ given $X$
and $L=k$ is equal to zero.

Suppose that we are given a random sample $\{(B_{i}^{\top },X_{i}^{\top
},L_{i})^{\top }\}_{i=1}^{n}$ of $(B^{\top },X^{\top },L)^{\top }$.\footnote{%
In fact, the estimator allows that we do not observe the whole vector $B_{i}$%
, but observe only $B_{1i},...,B_{ki}$ whenever $L_{i}=k$.} Assume that $%
q_{k}(\tau |x)$ is $(r+1)$-times continuously differentiable with respect to 
$x$, where $r\geq 1$. We use a local polynomial estimator $\widehat{q}%
_{k}(\tau |x)$, similar to \citeasnoun{Chaudhuri:91a}. For $u\equiv
(u_{1},\ldots ,u_{d})$, a $d$-dimensional vector of nonnegative integers,
let $[u]=u_{1}+\cdots +u_{d}$. Let $A_{r}$ be the set of all $d$-dimensional
vectors $u$ such that $[u]\leq r$, and let $|A_{r}|$ denote the number of
elements in $A_{r}$. For $z=(z_{1},...,z_{d})^{\top }\in \mathbf{R}^{d}$
with $u=(u_{1},...,u_{d})^{\top }\in A_{r}$, let $z^{u}=%
\prod_{m=1}^{d}z_{m}^{u_{m}}$. Now define $c(z)=(z^{u})_{u\in A_{r}},$ for $%
z\in \mathbf{R}^{d}$. Note that $c(z)$ is a vector of dimension $|A_{r}|$.
For $u=(u_{1},...,u_{d})^{\top }\in A_{r}$, and $r+1$ times differentiable
map $f$ on $\mathbf{R}^{d}$, we define the following derivative:%
\begin{equation*}
(D^{u}f)(x)\equiv \frac{\partial ^{\lbrack u]}}{\partial x_{1}^{u_{1}}\cdot
\cdot \cdot \partial x_{d}^{u_{d}}}f(x),
\end{equation*}%
where $[u]=u_{1}+\cdot \cdot \cdot +u_{d}$. Then we define $\gamma _{\tau
,k}(x)\equiv \left( \gamma _{\tau ,k,u}(x)\right) _{u\in A_{r}}$, where%
\begin{equation*}
\gamma _{\tau ,k,u}(x)\equiv \frac{1}{u_{1}!\cdot \cdot \cdot u_{d}!}%
D^{u}q_{k}(\tau |x).
\end{equation*}

We construct an estimator $\hat{\gamma}_{\tau ,k}(x)$ as follows. First, we
define for each $\gamma \in \mathbf{R}^{|A_{r}|},$ 
\begin{equation*}
S_{n,x,\tau ,k}(\gamma )\equiv \sum_{i=1}^{n}1\{L_{i}=k\}\sum_{\ell
=1}^{L_{i}}l_{\tau }\left[ B_{\ell i}-\gamma ^{\top }c\left( \frac{X_{i}-x}{h%
}\right) \right] K\left( \frac{x-X_{i}}{h}\right) .
\end{equation*}%
Then we construct%
\begin{equation}
\hat{\gamma}_{\tau ,k}(x)\equiv \text{argmin}_{\gamma \in \mathbf{R}%
^{|A_{r}|}}S_{n,x,\tau ,k}(\gamma ),  \label{argmin}
\end{equation}%
where $l_{\tau }(u)\equiv u[\tau -1\{u\leq 0\}]$ for any $u\in \mathbf{R}$, $%
K_{h}(t)=K(t/h)/h^{d}$, $K$ is a $d$-variate kernel function, and $h$ is a
bandwidth that goes to zero as $n\rightarrow \infty $.

In order to reduce the redundancy of the statements, let us introduce the
following definitions.

\begin{defn}
Let $\mathcal{G}$ be a set of functions $g_{v}:\mathbf{R}^{m}\rightarrow 
\mathbf{R}^{s}\ $indexed by a set $\mathcal{V}$, and let $S\subset \mathbf{R}%
^{m}$ be a given set and for $\varepsilon >0$, let $S_{v}(\varepsilon )$ be
an $\varepsilon $-enlargement of $S_{v}=\{x\in S:(x,v)\in S\times V\}$,
i.e., $S_{v}(\varepsilon )=\{x+a:x\in S$ and $a\in \lbrack -\varepsilon
,\varepsilon ]^{m}\}.$ Then we define the following conditions for $\mathcal{%
G}$:\medskip \newline
\noindent (a) B($S,\varepsilon $): $g_{v}$ is bounded on $S_{v}(\varepsilon
)\ $uniformly over $v\in \mathcal{V}.$\newline
\noindent (b) BZ$(S,\varepsilon $): $g_{v}$ is bounded away from zero on $%
S_{v}(\varepsilon )\ $uniformly over $v\in \mathcal{V}.$\newline
\noindent (c) BD($S,\varepsilon ,r$): $\mathcal{G}$ satisfies B$%
(S,\varepsilon )$ and $g_{v}$ is $r$ times continuously differentiable on $%
S_{v}(\varepsilon )$ with derivatives bounded on $S_{v}(\varepsilon )$
uniformly over $v\in \mathcal{V}.$\newline
\noindent (d) BZD($S,\varepsilon ,r$): $\mathcal{G}$ satisfies BZ$%
(S,\varepsilon )$ and $g_{v}$ is $r$ times continuously differentiable on $%
S_{v}(\varepsilon )$ with derivatives bounded on $S_{v}(\varepsilon )$
uniformly over $v\in \mathcal{V}.$\newline
\noindent (e) LC: $g_{v}$ is Lipschitz continuous with Lipschitz coefficient
bounded uniformly over $v\in V$.
\end{defn}

Let $\mathcal{P}$ denote the collection of the potential joint distributions
of $(B^{\top },X^{\top },L)^{\top }$ and define $\mathcal{V}=\mathcal{T}%
\times \mathcal{P}$, and for each $k\in \mathbb{N}_{L},$%
\begin{eqnarray}
\mathcal{G}_{q}(k) &=&\left\{ q_{k}(\tau |\cdot ):(\tau ,P)\in \mathcal{V}%
\right\} ,  \label{fns} \\
\mathcal{G}_{f}(k) &=&\left\{ f_{\tau ,k}(\cdot |\cdot ):(\tau ,P)\in 
\mathcal{V}\right\} ,  \notag \\
\mathcal{G}_{L}(k) &=&\left\{ P\left\{ L_{i}=k|X_{i}=\cdot \right\} :P\in 
\mathcal{P}\right\} ,\text{ and}  \notag \\
\mathcal{G}_{f} &=&\left\{ f(\cdot ):P\in \mathcal{P}\right\} ,  \notag
\end{eqnarray}%
where $f_{\tau ,k}(0|x)$ being the conditional density of $B_{li}-q_{k}(\tau
|X_{i})$ given $X_{i}=x$ and $L_{i}=k.$ Also, define%
\begin{equation}
\mathcal{G}_{f,2}(k)=\left\{ f_{\cdot ,k}(\cdot |\cdot ):P\in \mathcal{P}%
\right\} \text{ and\ }\mathcal{G}_{\gamma }(k)=\left\{ \gamma _{\cdot
,k}(\cdot ):P\in \mathcal{P}\right\} .  \label{fns2}
\end{equation}
In other words, $\mathcal{G}_{f,2}(k)$ is the class of conditional densities $f_{\tau,k}(\cdot|x)$ indexed by $\tau$, $x$, and probabilities $P$, and $\mathcal{G}_\gamma(k)$ is the class of functions $\gamma_{\tau,k}(\cdot)$ indexed by $\tau$ and probabilities $P$. We make the following assumptions.

\begin{AssumptionQR}
(i) $\mathcal{G}_{f}$ satisfies BD($\mathcal{S},\varepsilon ,1)$.\newline
\noindent (ii) For each $k\in \mathbb{N}_{L}$, $\mathcal{G}_{f}(k)$ and $%
\mathcal{G}_{L}(k)\ $satisfy both BD($\mathcal{S},\varepsilon ,1)$ and BZD($%
\mathcal{S},\varepsilon ,1).$\newline
\noindent (iii) For each $k\in \mathbb{N}_{L}$, $\mathcal{G}_{q}(k)\ $%
satisfies BD($\mathcal{S},\varepsilon ,r+1)$ for some $r\geq 1.$\newline
\noindent (iv) For each $k\in \mathbb{N}_{L}$, $\mathcal{G}_{f,2}(k)$ and $%
\mathcal{G}_{\gamma }(k)\ $satisfy LC.
\end{AssumptionQR}

Assumptions QR1(i) and (iii) are standard assumptions used in the local
polynomial approach where one approximates $q_{k}(\cdot |x)$ by a linear
combination of its derivatives through Taylor expansion, except only that
the approximation here is required to behave well uniformly over $P\in 
\mathcal{P}$. Assumption QR1(ii) is made to prevent the degeneracy of the
asymptotic linear representation of $\hat{\gamma}_{\tau ,k}(x)-\gamma _{\tau
,k}(x)$ that is uniform over $x\in \mathcal{S}_{\tau }(\varepsilon ),\ \tau
\in \mathcal{T}$ and over $P\in \mathcal{P}$. Assumption QR1(iv) requires
that the conditional density function of $B_{li}-q_{k}(\tau |X_{i})$ given $%
X_{i}=x$ and $L_{i}=k$ and $\gamma _{\tau ,k}(\cdot )$ behave smoothly as we
perturb $\tau $ locally. This requirement is used to control the size of the
function spaces indexed by $\tau $, so that when the stochastic convergence
of random sequences holds, it is ensured to hold uniformly in $\tau $.

Let $||\cdot ||$ denote the Euclidean norm throughout the paper. Assumption
QR2 lists conditions for the kernel function and the bandwidth.

\begin{AssumptionQR}
(i) $K$ is compact-supported, bounded, and Lipschitz continuous
on the interior of its support, $\int K(u)du=1$, and $\int K\left( u\right)
||u||^{2}du>0.$\newline
(ii) As $n\rightarrow \infty ,\ n^{-1/2}h^{-d/2}\log n+\sqrt{nh^{d}}h^{r+1}/%
\sqrt{\log n}\rightarrow 0,$ with $r$ in Assumption QR1(iii).
\end{AssumptionQR}
Assumption QR2 gives conditions for the kernel and the bandwidth. The condition for the bandwidth is satisfied if we take $h = Cn^{-s}$ for some constant $C$ with $s>0$ satisfying that $1/(d+2(r+1))<s<1/d$.
  
For any sequence of real numbers $b_{n}>0$, and any sequence of random
vectors $Z_{n} $, we say that $Z_{n}/b_{n}\rightarrow _{P}0,$ $\mathcal{P}$%
-uniformly, or $Z_{n}=o_{P}(b_{n}),$ $\mathcal{P}$-uniformly, if for any $%
a>0 $, 
\begin{equation*}
\underset{n\rightarrow \infty }{\limsup}\sup_{P\in \mathcal{P}%
}P\left\{ ||Z_{n}||>ab_{n}\right\} =0\text{.}
\end{equation*}%
Similarly, we say that $Z_{n}=O_{P}(b_{n})$, $\mathcal{P}$-uniformly, if for
any $a>0$, there exists $M>0$ such that 
\begin{equation*}
\underset{n\rightarrow \infty }{\limsup}\sup_{P\in \mathcal{P}%
}P\left\{ ||Z_{n}||>Mb_{n}\right\} <a\text{.}
\end{equation*}

Below, we establish a uniform Bahardur representation of $\sqrt{nh^{d}}H(%
\hat{\gamma}_{\tau ,k}(x)-\gamma _{\tau ,k}(x))$, where $H=\ $diag$%
((h^{|u|})_{u\in A_{r}})$ is the $|A_{r}|\times |A_{r}|$ diagonal
matrix. First we introduce some notation. We define%
\begin{eqnarray*}
\Delta _{x,\tau ,lk,i} &\equiv &B_{li}-\gamma _{\tau ,k}^{\top
}(x)c(X_{i}-x), \\
c_{h,x,i} &\equiv &c\left( (X_{i}-x)/h\right) ,\text{ and }K_{h,x,i}\equiv
K\left( (X_{i}-x)/h\right) .
\end{eqnarray*}%
Let%
\begin{equation*}
M_{n,\tau ,k}(x)\equiv k\int P\{L_{i}=k|X_{i}=x+th\}f_{\tau
,k}(0|x+th)f(x+th)K(t)c(t)c^{\top }(t)dt.
\end{equation*}

\begin{thm}
\textit{Suppose that Assumptions QR1-QR2 hold. Then, for each }$k\in \mathbb{%
N}_{L},$%
\begin{eqnarray*}
&&\sup_{\tau \in \mathcal{T},x\in \mathcal{S}_{\tau }(\varepsilon
)}\left\Vert \sqrt{nh^{d}}H(\hat{\gamma}_{\tau ,k}(x)-\gamma _{\tau
,k}(x))-M_{n,\tau ,k}^{-1}(x)\left(\psi _{n,x,\tau ,k}-\mathbf{E}\psi _{n,x,\tau ,k}\right)\right\Vert \\
&=&O_{P}\left( \frac{\log ^{1/2}n}{n^{1/4}h^{d/4}}\right) ,\ \mathcal{P}%
\text{\textit{-uniformly},}
\end{eqnarray*}%
\textit{where, with }$\tilde{l}_{\tau }\left( x\right) \equiv \tau -1\{x\leq
0\}$,%
\begin{equation*}
\psi _{n,x,\tau ,k}\equiv -\frac{1}{\sqrt{nh^{d}}}\sum_{i=1}^{n}1\left\{
L_{i}=k\right\} \sum_{l=1}^{L_{i}}\tilde{l}_{\tau }\left( \Delta _{x,\tau
,lk,i}\right) c_{h,x,i}K_{h,x,i}\text{.}
\end{equation*}
\end{thm}

%... However, to the best of our knowledge, there has not appeared a formal Bahadur representation result for local polynomial estimators of nonparametric quantile regression models that is uniform over a wide class of probabilities. (THIS\ LATTER\ STATEMENT\ DOES\ NOT\ SEEM\ RIGHT.) Our results are built on the convexity arguments in quantile regression models in Pollard (1990) (see Kato (2009) for a related recent result) and the construction of bracketing entropy bound in a manner similar to \citeasnoun{Guerre/Sabbah:12} which is used to apply a maximal inequality of Massart (2007). (See also a recent contribution by \citeasnoun{Qu:Yoon:15}.) Nevertheless, the uniform error rate obtained here is slightly sharper than that of \citeasnoun{Guerre/Sabbah:12}.

The proof in this paper uses the convexity arguments of %
\citeasnoun{Pollard:91} (see \citeasnoun{Kato:09} for a related recent
result) and, similarly as in \citeasnoun{Guerre/Sabbah:12}, employs the
maximal inequality of \citeasnoun[Theorem 6.8]{Massart:07}. The theoretical
innovation of Theorem 1 is that we have obtained an approximation that is
uniform in $(x,\tau )$ as well as in $P$. See Remark 1 below for a detailed
comparison.

%To the best of our knowledge, there is no established result on linear
%expansions of local polynomial quantile regression estimators that hold
%uniformly over three aspects $(x,\tau ,P)$ simultaneously. 

\begin{rem}
The main difference between this paper and \citeasnoun{Guerre/Sabbah:12} is
that their result pays attention to uniformity in $h$ over some range, while
our result pays attention to uniformity in $P$. Also it is interesting to
note that the error rate here is a slight improvement over theirs. When $d=1$%
, the rate here is $O_{P}[\sqrt{\log n}/(n^{1/4}h^{1/4})]$ while the rate in
Theorem 2 of \citeasnoun{Guerre/Sabbah:12} is $O_{P}[\log
^{3/4}n/(n^{1/4}h^{1/4})]$. The difference is due to our use of an improved
inequality which leads to a tighter bound in the maximal inequality in
deriving the uniform error rate. For details, see the remark after the proof
of Theorem 1 in the appendix.
\end{rem}

The summands in the asymptotic linear representation form in Theorem 1
depend on the sample size and are not centered conditional on $X_i$. While this form can be useful in some contexts, the form is less illuminating. We provide an asymptotic linear representation that ensures this conditional centering given $X_i$.

\begin{cor}
\textit{Suppose that Assumptions QR1-QR2 hold. Then, for each }$k\in \mathbb{%
N}_{L},$%
\begin{eqnarray*}
&&\sup_{\tau \in \mathcal{T},x\in \mathcal{S}_{\tau }(\varepsilon
)}\left\Vert \sqrt{nh^{d}}H(\hat{\gamma}_{\tau ,k}(x)-\gamma _{\tau
,k}(x))-M_{n,\tau ,k}^{-1}(x)\tilde{\psi}_{n,x,\tau ,k}\right\Vert \\
&=&O_{P}\left( \frac{\log ^{1/2}n}{n^{1/4}h^{d/4}}\right) ,\ \mathcal{P}%
\text{\textit{-uniformly},}
\end{eqnarray*}%
\textit{where}%
\begin{equation*}
\tilde{\psi}_{n,x,\tau ,k}\equiv -\frac{1}{\sqrt{nh^{d}}}\sum_{i=1}^{n}1%
\left\{ L_{i}=k\right\} \sum_{l=1}^{L_{i}}\tilde{l}_{\tau }\left(
\varepsilon _{\tau ,lk,i}\right) c_{h,x,i}K_{h,x,i}\text{.}
\end{equation*}%
\bigskip
\end{cor}
Note that the quantity $M_{n,\tau,k}(x)$ in the representation can be replaced by
\begin{equation*}
M_{\tau ,k}(x)\equiv k\int P\{L_{i}=k|X_{i}=x\}f_{\tau
	,k}(0|x)f(x)K(t)c(t)c^{\top }(t)dt,
\end{equation*}
if we modify the conditions on kernels and the smoothness conditions for the nonparametric function $P\{L_{i}=k|X_{i}=x\}f_{\tau,k}(0|x)f(x)$. As this modification can be done in a standard manner, we do not pursue details.

\subsection{Bootstrap Uniform Bahadur Representation for Local Polynomial
Quantile Regression Estimator}

Let us consider the bootstrap version of the Bahadur representation in Theorem 1. Suppose that $\{(Y_{i}^{\ast \top
},X_{i}^{\ast \top })\}_{i=1}^{n}$ is a bootstrap sample drawn with
replacement from the empirical distribution of $\{(Y_{i}^{\top },X_{i}^{\top
})_{i=1}^{n}\}$. Throughout the paper, the bootstrap distribution $P^{\ast
} $ is viewed as the distribution of $(Y_{i}^{\ast },X_{i}^{\ast })_{i=1}^{n},$ conditional on $(Y_{i},X_{i})_{i=1}^{n}$, and let $\mathbf{E}^*$ be expectation with respect to $P^*$.

We define the notion of uniformity in the convergence of distributions under 
$P^{\ast }$. For any sequence of real numbers $b_{n}>0$, and any sequence of
random vectors $Z_{n}^{\ast }$, we say that $Z_{n}^{\ast }/b_{n}\rightarrow
_{P^{\ast }}0,$ $\mathcal{P}$-uniformly, or $Z_{n}^{\ast }=o_{P^{\ast
}}(b_{n}),$ $\mathcal{P}$-uniformly, if for any $a>0$,%
\begin{equation*}
\underset{n\rightarrow \infty }{\limsup}\sup_{P\in \mathcal{P}%
}P\left\{ P^{\ast }\left\{ ||Z_{n}^{\ast }||>ab_{n}\right\} >a\right\} =0%
\text{.}
\end{equation*}%
Similarly, we say that $Z_{n}^{\ast }=O_{P^{\ast }}(b_{n})$, $\mathcal{P}$%
-uniformly, if for any $a>0$, there exists $M>0$ such that 
\begin{equation*}
\underset{n\rightarrow \infty }{\limsup}\sup_{P\in \mathcal{P}%
}P\left\{ P^{\ast }\left\{ ||Z_{n}^{\ast }||>Mb_{n}\right\} >a\right\} <a%
\text{.}
\end{equation*}

For $z=(x,\tau )\in \mathcal{Z}$, define $\Delta _{x,\tau ,lk,i}^{\ast
}\equiv Y_{l,i}^{\ast }-\gamma _{\tau ,k}^{\top }(x)c(X_{i}^{\ast }-x)$, and
let $c_{h,x,i}^{\ast }$ and $K_{h,x,i}^{\ast }$ be $c_{h,x,i}$ and $%
K_{h,x,i} $ except that $X_{i}$ is replaced by $X_{i}^{\ast }$. Then the
following theorem gives the bootstrap version of Theorem 1.

\begin{thm}
\textit{Suppose that Assumptions QR1-QR2 hold. Then for each} $k\in \mathbb{N%
}_{J},$%
\begin{eqnarray*}
&&\sup_{(x,\tau )\in \mathcal{X}_{1}\times \mathcal{T}}\left\Vert \sqrt{%
nh^{d}}H(\hat{\gamma}_{\tau ,k}^{\ast }(x)-\hat{\gamma}_{\tau
,k}(x))-M_{n,\tau ,k}^{-1}(x)\left(\psi _{n,x,\tau ,k}^* - \mathbf{E}^*\psi _{n,x,\tau ,k}^* \right) \right\Vert \\
&=&O_{P^{\ast }}\left( \frac{\log ^{1/2}n}{n^{1/4}h^{d/4}}\right) ,\text{ }%
\mathcal{P}\text{\textit{-uniformly,}}
\end{eqnarray*}%
\textit{where }$\psi _{n,x,\tau ,k}^{\ast }\equiv -\frac{1}{\sqrt{nh^{d}}}%
\sum_{i=1}^{n}1\{L_{i}=k\}\sum_{l=1}^{k}\tilde{l}_{\tau }\left( \Delta
_{x,\tau ,lk,i}^{\ast }\right) c_{h,x,i}^{\ast } K_{h,x,i}^{\ast
} $.
\end{thm}

Theorem 2 is obtained by using Le Cam's Poissonization Lemma (see %
\citeasnoun[Proposition 2.5]{Gine:97}) and following the proof of Theorem 1.
The bootstrap version of Corollary 1 follows immediately from Theorem 2.

\section{Testing Monotonicity of Quantile Regression}

\label{test-mono-qr}

This section considers testing monotonicity of quantile regression. We first
state the testing problem formally, give the form of test statistic, verify
regularity conditions, and present results of simple Monte Carlo experiments.

\subsection{Testing Problem}

\label{test:prob:qr:mono}

Let $q(\tau |x)$ denote the $\tau $-th quantile of $Y$ conditional on $X=x$,
where $\tau \in (0,1)$ and $X$ is a scalar random variable. In this subsection, we consider testing monotonicity of quantile regression. Define $g_{\tau }(x)\equiv \partial q(\tau|x)/\partial x$. The null hypothesis and the alternative hypothesis are as follows: 
\begin{align}
H_{0}& :g_{\tau }(x)\leq 0\ \text{for all }(\tau ,x)\in \mathcal{T\times X}%
\text{ against}  \label{null1} \\
H_{1}& :g_{\tau }(x)>0\ \text{for some }(\tau ,x)\in \mathcal{T}\times 
\mathcal{X},  \notag
\end{align}%
where $\mathcal{X}$ is contained in the support of $X$ and $\mathcal{T}%
\subset (0,1)$. The null hypothesis states that the quantile functions are
increasing on $\mathcal{X}$ for all $\tau \in \mathcal{T}$, and the
alternative hypothesis is the negation of the hypothesis. If $\mathcal{T}$
is a singleton, then testing \eqref{null1} amounts to testing
monotonicity of quantile regression at a fixed quantile.

Suppose that $q(\tau |x)$ is continuously differentiable on $\mathcal{X}$
for each $\tau \in \mathcal{T}$. Then one natural approach is to test the
sign restriction of the derivative of $q(\tau |x)$. In other words, we 
develop a test of inequality restrictions using the local polynomial estimator of $\partial q(\tau
|x)/\partial x$.

One may consider various other forms of monotonicity tests for quantile
regression. For example, one might be interested in monotonicity of an
interquartile regression function. More specifically, let $\tau _{1}<\tau
_{2}$ be chosen from $(0,1)$ and write $\Delta g_{\tau _{1},\tau
_{2}}(x)\equiv g_{\tau _{2}}(x)-g_{\tau _{1}}(x)$. Then the null hypothesis
and the alternative hypothesis of monotonicity of the interquartile
regression function are as follows:%
\begin{eqnarray}
H_{0,\Delta } &:&\Delta g_{\tau _{1},\tau _{2}}(x)\leq 0\ \text{for all }%
x\in \mathcal{X}\text{ against}  \label{null3} \\
H_{1,\Delta } &:&\Delta g_{\tau _{1},\tau _{2}}(x)>0\ \text{for some }x\in 
\mathcal{X}\text{.}  \notag
\end{eqnarray}%
The null hypothesis states that the interquartile regression function $%
q(\tau _{2}|x)-q(\tau _{1}|x)$ is increasing on $\mathcal{X}$. This type of
monotonicity can be used to investigate whether the income inequality (in
terms of interquartile comparison) becomes severe as certain demographic
variable $X$ such as age increases.

\subsection{Test Statistic}

Suppose that we are given a random sample $\{(Y_{i},X_{i})\}_{i=1}^{n}$ of $%
(Y,X)$. First, we estimate $g_{\tau }(x)$ by local polynomial estimation to obtain, say, $%
\hat{g}_{\tau }(x) \equiv \mathbf{e}%
_{2}^{\top }\mathbf{\hat{\gamma}}_{\tau }(x)$, where 
$\mathbf{e}_2$ is a column vector whose second entry is one and the rest zero, and 
\begin{equation}
\mathbf{\hat{\gamma}}_{\tau }(x)\equiv \text{argmin}_{\mathbf{\gamma
\in R}^{r+1}}\sum_{i=1}^{n}l_{\tau }\left( Y_{i}-\mathbf{\gamma }%
^{\top }c(X_{i}-x)\right) K_{h}(X_{i}-x).
\label{gamma_hat_original}
\end{equation}%

This paper applies the general approach of LSW, and proposes a new nonparametric test of mononoticity hypotheses involving quantile regression functions. First, testing the monotonicity of $q(\tau |\cdot )$ is tantamount to testing nonnegativity of $g_{\tau }$ on a domain of interest. Define $\Lambda _{p}(a)=\left( \max \left\{ a,0\right\} \right) ^{p}$ for
any real numbers $a>0$ and $p\geq 1.$
We consider two test statistics corresponding to the two testing problems 
discussed in Section \ref{test:prob:qr:mono}: for $1 \leq p < \infty,$%
\begin{eqnarray}
T_{n,2} &\equiv &\int_{\mathcal{X}\times \mathcal{T}}\Lambda _{p}(\hat{g}%
_{\tau }(x))w_{\tau }(x)d(x,\tau )\text{, and}  \label{test stats} \\
T_{n}^{\Delta } &\equiv &\int_{\mathcal{X}}\Lambda _{p}(\Delta \hat{g}_{\tau
_{1},\tau _{2}}(x))w(x)dx,\text{ }\tau _{1},\tau _{2}\in (0,1),  \notag
\end{eqnarray}%
where $\Delta \hat{g}_{\tau _{1},\tau _{2}}(x)\equiv \hat{g}_{\tau
_{2}}(x_{1})-\hat{g}_{\tau _{1}}(x_{1})$, $w_{\tau }(\cdot )$ and $w(\cdot )$
are nonnegative weight functions. The test statistic $T_{n,2}$ is used to
test $H_{0}$ against $H_{1}$ and the test statistic $T_{n,\Delta }$ to test $%
H_{0,\Delta }$ against $H_{1,\Delta }$.

Let $\{(Y_{i}^{\ast },X_{i}^{\ast })\}$ be a bootstrap sample obtained from
resampling from $\{(Y_{i},X_{i})\}_{i=1}^{n}$ with replacement. Then we define%
\begin{equation}
\mathbf{\hat{\gamma}}_{\tau }^{\ast }(x)\equiv \text{argmin}_{\mathbf{\gamma
\in R}^{r+1}}\sum_{i=1}^{n}l_{\tau }\left( Y_{i}^{\ast }-\mathbf{\gamma }%
^{\top }c(X_{i}^{\ast }-x)\right) K_{h}(X_{i}^{\ast }-x)
\label{gamma_hat}
\end{equation}%
and take  $\hat{g}_{\tau }^{\ast }(x)\equiv \mathbf{e}%
_{2}^{\top }\mathbf{\hat{\gamma}}_{\tau }^{\ast }(x)$, similarly as before.
We construct the \textquotedblleft recentered\textquotedblright\ bootstrap
test statistics:%
\begin{eqnarray}
T_{n,2}^{\ast } &\equiv &\int_{\mathcal{X}\times \mathcal{T}}\Lambda _{p}\left( \hat{g%
}_{\tau }^{\ast }(x)-\hat{g}_{\tau }(x)\right) w_{\tau }(x)d(x,\tau ),\text{
and}  \label{bt test statistics} \\
T_{n,\Delta }^{\ast } &\equiv &\int_{\mathcal{X}}\Lambda
_{p}\left( \Delta \hat{g}_{\tau _{1},\tau _{2}}^{\ast }(x)-\Delta \hat{g}%
_{\tau _{1},\tau _{2}}(x)\right) w(x)dx,  \notag
\end{eqnarray}%
where $\Delta \hat{g}_{\tau _{1},\tau _{2}}^{\ast }(x)\equiv \hat{g}_{\tau
_{2}}^{\ast }(x)-\hat{g}_{\tau _{1}}^{\ast }(x)$. We can now take the
bootstrap critical values $c_{2,\alpha }^{\ast }$ and $c_{\Delta ,\alpha
}^{\ast }$ to be the $(1-\alpha )$ quantiles from the bootstrap
distributions of $T_{n,2}^{\ast }$ and $T_{n,\Delta }^{\ast }$. Then we
define 
\begin{equation*}
c_{2,\alpha ,\eta }^{\ast }=\max \{c_{2,\alpha }^{\ast },h^{1/2}\eta +\hat{a}%
_{2}^{\ast }\}\text{ and\ }c_{\Delta ,\alpha ,\eta }^{\ast }=\max
\{c_{\Delta ,\alpha }^{\ast },h^{1/2}\eta +\hat{a}_{\Delta }^{\ast }\},
\end{equation*}%
where$\ \hat{a}_{2}^{\ast }\equiv \mathbf{E}^{\ast }T_{n,2}^{\ast }$ and $%
\hat{a}_{\Delta }^{\ast }\equiv \mathbf{E}^{\ast }T_{n,\Delta }^{\ast }$.
The $(1-\alpha )$-level bootstrap tests for the two hypotheses are defined as%
\begin{equation}
\begin{tabular}{ll}
$\text{Reject }H_{0}$ & $\text{if and only if }T_{n,2}>c_{2,\alpha ,\eta
}^{\ast }$. \\ 
$\text{Reject }H_{0,\Delta }$ & $\text{if and only if }T_{n,\Delta
}>c_{\Delta ,\alpha ,\eta }^{\ast }.$%
\end{tabular}
\label{tests}
\end{equation}

\subsection{Primitive Conditions}

We present primitive conditions for the asymptotic validity of the proposed monotonicity tests. Let $\mathcal{P}$
denote the collection of the potential joint distributions of $(Y,X)^{\top }$
and define $\mathcal{V}=\mathcal{T}\times \mathcal{P}$ as before. We also
define $\mathcal{\tilde{V}}=\mathcal{T}^{2}\times \mathcal{P}$ and define $%
\mathcal{G}_{f}$ as in (\ref{fns}). Similarly as in (\ref{fns}), we
introduce the following definitions: 
\begin{eqnarray*}
\mathcal{G}_{g} &=&\left\{ g_{\tau }(\cdot ):(\tau ,P)\in \mathcal{V}%
\right\} , \\
\mathcal{G}_{\Delta g} &=&\{\Delta g_{\tau _{1},\tau _{2}}(\cdot ):(\tau
_{1},\tau _{2},P)\in \mathcal{\tilde{V}\}},\text{ and} \\
\mathcal{G}_{Q,f} &=&\left\{ f_{\tau }(\cdot |\cdot ):(\tau ,P)\in \mathcal{V%
}\right\} ,
\end{eqnarray*}%
where $f_{\tau }(0|x)$ being the conditional density of $Y_{i}-q(\tau
|X_{i}) $ given $X_{i}=x.$ Also, define%
\begin{equation*}
\mathcal{G}_{f,2}=\left\{ f_{\cdot }(\cdot |\cdot ):P\in \mathcal{P}\right\} 
\text{ and\ }\mathcal{G}_{\gamma }=\left\{ \gamma _{\cdot }(\cdot ):P\in 
\mathcal{P}\right\} .
\end{equation*}%
We make the following assumptions.
\begin{AssumptionMON}
(i) $\mathcal{G}_{f}$ satisfies BD($\mathcal{S},\varepsilon ,1)$.\newline
\noindent (ii) $\mathcal{G}_{Q,f}\ $satisfies both BD($\mathcal{S},\varepsilon ,1$%
) and BZD($\mathcal{S},\varepsilon ,1).$\newline
\noindent (iii) $\mathcal{G}_{g}\ $satisfies BD($\mathcal{S},\varepsilon
,r+1 $) for some $r>3/2.$\newline
\noindent (iv) $\mathcal{G}_{f,2}$ and $\mathcal{G}_{\gamma }$ satisfy LC.
\end{AssumptionMON}

\begin{AssumptionMON}
(i) $K$ is nonnegative and satisfies Assumption QR2(i).\newline
\noindent (ii) $n^{-1/2}h^{-\{(3+\nu )/2\}+1}+\sqrt{n}h^{r+2}/\sqrt{\log n}%
\rightarrow 0,$ as $n\rightarrow \infty ,\ $for some small $\nu >0$, with $r$
in Assumption MON1(iii).
\end{AssumptionMON}

Assumption MON1 introduces a set of regularity conditions for various
function spaces. Conditions (i)-(iii) require smoothness conditions. In
particular, Condition (iii) is used to control the bias of the nonparametric
quantile regression derivative estimator. Condition (iv) is analogous to
Assumption QR1(iv) and used to control the size of the function space
properly. Assumption MON2 introduces conditions for the kernel and bandwidth. The condition in Assumption MON2(ii) requires a bandwidth condition that is stronger than that in Assumption QR2(ii).

Assumptions IQM1-IQM2 below are used for testing $
H_{0,\Delta }$.

\begin{AssumptionIQM}
(i) Assumptions MON1(i),(ii) and (iv) hold.\newline
\noindent (ii) $\mathcal{G}_{\Delta g}\ $satisfies BD($\mathcal{S}%
,\varepsilon ,r+1)$ for some $r>3/2.$
\end{AssumptionIQM}

\begin{AssumptionIQM}
The kernel function $K$ and the bandwidth $h$ satisfy Assumption MON2.
\end{AssumptionIQM}

The following result establishes the uniform validity of the bootstrap test.
Let $\mathcal{P}_{0} \subset \mathcal{P}$ denote the set of potential distributions of the observed random vector under the null
hypothesis.

\begin{thm}\label{Thm3}
(i) Suppose that Assumptions MON1-MON2 hold. Then,
\begin{equation*}
\limsup_{n\rightarrow \infty }\sup_{P\in \mathcal{P}%
_{0}}P\{T_{n,2}>c_{2,\alpha ,\eta }^{\ast }\}\leq \alpha .
\end{equation*}%
\newline
(ii) Suppose that Assumptions IQM1-IQM2 hold. Then,
\begin{equation*}
\limsup_{n\rightarrow \infty }\sup_{P\in \mathcal{P}_{0}}P\left\{
T_{n,\Delta }>c_{\Delta ,\alpha ,\eta }^{\ast }\right\} \leq \alpha .
\end{equation*}
\end{thm}

Using the general framework of LSW, it is possible to establish the consistency and local power of the test.
Furthermore, it is also feasible to obtain a more powerful (but still asymptotically uniformly valid) test by estimating a contact set
at the expense of requiring an additional tuning parameter (see LSW for details). 

\subsection{Monte Carlo Experiments}

In this subsection, we present results of some Monte Carlo experiments that
illustrate the finite-sample performance of one of the proposed tests.
Specifically, we consider the following null and alternative hypotheses: 
\begin{equation*}
H_{0}:\forall x\in \mathcal{X},\ g(x)\geq 0\;vs.\;H_{1}:\exists x\in 
\mathcal{X},\ g(x)<0,
\end{equation*}%
where $g(x)\equiv \partial \mathrm{Median}\left[ Y|X=x\right] /\partial x$.
In the experiments, $X$ is generated independently from $\text{Unif}[0,1]$
and $U$ follows the distribution of $X^{4}\times \mathbf{N}(0,0.1)$.

To check the size of the test, we generated $Y = U$, which we call the null
model. Note that the null model corresponds to the least favorable case. To
see the power of the test, we considered the following alternative models: $%
Y=m_{j}(X)+U$ $(j=1,2,3,4,5)$, where 
\begin{eqnarray*}
m_{1}(x) &=&x(1-x), \\
m_{2}(x) &=&-0.1x, \\
m_{3}(x) &=&-0.1\exp (-50(x-0.5)^{2}), \\
m_{4}(x) &=& x + 0.6\exp(-10x^2), \\
m_{5}(x) &=& [10(x-0.5)^3 - 2\exp(-10(x-0.5)^2)]1( x < 0.5) \\
& & + [0.1(x-0.5) - 2\exp(-10(x-0.5)^2)]1(x \geq 0.5).
\end{eqnarray*}
In all experiments, $\mathcal{X}=[0.05,0.95]$. Figures \ref{figure1} and \ref%
{figure2} show the true functions and corresponding simulated data.

The experiments use sample size of $n=200$ and the nominal level of $\alpha
=0.10,0.05,$ and $0.01.$ We performed $1,000$ Monte Carlo replications for
the null model and $200$ replications for the alternative models. For each
replication, we generated $200$ bootstrap resamples. We used the local
linear quantile regression estimator with the uniform kernel on $[-1/2,1/2]\ 
$for $K(\cdot )$. Furthermore, for the test statistic, we used $p=2$
(one-sided $L_2$ norm) and uniform weight function $w(x)=1$ and $h \in
\{0.9,1,1.1\}$.

For the null model, the bootstrap approximation is quite good, especially
with $h = 1$. The test shows good power for alternative models 1-3 and 5.
For the alternative model 4, the power is sensitive with respect to the
choice of the bandwidth. Overall, the finite sample behavior of the proposed
test is satisfactory.

\section{Conclusions}

\label{sec:conclusion}

In this paper, we have established a uniform error rate of a Bahadur
representation for the local polynomial quantile regression estimator. The
error rate is uniform over a range of quantiles, a range of evaluation
points in the regressors, and over a wide class of probabilities for
observed random variables. We have illustrated the use of our Bahadur
representation in the context of testing monotonicity. In
addition, we have established the same error rate for the bootstrap local
polynomial quantile regression estimator, which can be useful for bootstrap inference.

\section{Proofs}

\label{sec:proofs}

We also define for $a,b\in \mathbf{R}^{|A_{r}|},$%
\begin{equation*}
\zeta _{n,x,\tau ,k}(a,b)\equiv \sum_{i=1}^{n}1\left\{ L_{i}=k\right\}
\sum_{l=1}^{L_{i}}\left\{ 
\begin{array}{c}
l_{\tau }\left( \Delta _{x,\tau ,lk,i}-(a+b)^{\top }c_{h,x,i}/\sqrt{nh^{d}}%
\right) \\ 
-l_{\tau }\left( \Delta _{x,\tau ,lk,i}-a^{\top }c_{h,x,i}/\sqrt{nh^{d}}%
\right)%
\end{array}%
\right\} K_{h,x,i},
\end{equation*}%
and 
\begin{equation*}
\zeta _{n,x,\tau ,k}^{\Delta }(a,b)\equiv \zeta _{n,x,\tau ,k}(a,b)-b^{\top
}\psi _{n,x,\tau ,k}.
\end{equation*}

\begin{LemmaQR}
\textit{Suppose that Assumptions QR1-QR2 hold. Let }$\{\delta
_{1n}\}_{n=1}^{\infty }$\textit{\ and }$\{\delta _{2n}\}_{n=1}^{\infty }$%
\textit{\ be positive sequences such that }$\delta _{1n}=M\sqrt{\log n}$ 
\textit{for some }$M>0$ \textit{and }$\delta _{2n}\leq \delta _{1n}$\textit{%
\ from some large }$n$ \textit{on}. \textit{Then for each }$k\in \mathbb{N}%
_{L},$ \textit{the following holds uniformly over} $P\in \mathcal{P}$:

\noindent (i)%
\begin{eqnarray*}
&&\mathbf{E}\left[ \sup_{a,b:||a||\leq \delta _{1n},||b||\leq \delta
_{2n}}\sup_{\tau \in \mathcal{T},x\in \mathcal{S}_{\tau }(\varepsilon
)}|\zeta _{n,x,\tau ,k}^{\Delta }(a,b)-\mathbf{E}[\zeta _{n,x,\tau
,k}^{\Delta }(a,b)]|\right] \\
&=&O\left( \frac{\delta _{2n}\sqrt{\log n}}{n^{1/4}h^{d/4}}\right) .
\end{eqnarray*}

\noindent (ii)%
\begin{equation*}
\mathbf{E}\left[ \sup_{\tau \in \mathcal{T},x\in \mathcal{S}_{\tau
}(\varepsilon )}\left\Vert \psi _{n,x,\tau ,k}\right\Vert \right] =O\left( 
\sqrt{\log n}+h^{r+1}\right)=O\left( 
\sqrt{\log n}\right) .
\end{equation*}

\noindent (iii)%
\begin{eqnarray*}
&&\sup_{a,b:||a||\leq \delta _{1n},||b||\leq \delta _{2n}}\sup_{\tau \in 
\mathcal{T},x\in \mathcal{S}_{\tau }(\varepsilon )}\left\vert \mathbf{E}%
[\zeta _{n,x,\tau ,k}^{\Delta }(a,b)]-\frac{b^{\top }M_{n,\tau ,k}(x)(b+2a)}{%
2}\right\vert \\
&=&O\left( \frac{\delta _{2n}\delta _{1n}^{2}}{n^{1/2}h^{d/2}}\right) .
\end{eqnarray*}%
\textit{\ }
\end{LemmaQR}

\begin{proof}[Proof of Lemma QR1]
(i) Define%
\begin{equation}
\tilde{\delta}_{\tau ,k}(x_{1};x)\equiv q_{k}(\tau |x_{1})-\gamma _{\tau
,k}(x_{1})^{\top }c(x_{1}-x),  \label{delta_tilde}
\end{equation}%
and 
\begin{equation}
\delta _{n,\tau ,k}(x_{1};x)\equiv \tilde{\delta}_{\tau
,k}(x_{1};x)1\{|x_{1}-x|\leq h\},  \label{delta_tau}
\end{equation}%
where the dependence on $P$ is through $q_{k}(\tau |x_{1})$ and $\gamma
_{\tau ,k}(x_{1})$. We also let 
\begin{equation}
\delta _{n,\tau ,k}(x_{1})\equiv \sup_{x\in \mathcal{S}_{\tau }(\varepsilon
)}\sup_{P\in \mathcal{P}}|\delta _{n,\tau ,k}(x_{1};x)|.  \label{delta}
\end{equation}%
It is not hard to see that 
\begin{equation}
\text{sup}_{\tau \in \mathcal{T},\ x_{1}\in \mathcal{S}_{\tau }(\varepsilon
)}|\delta _{n,\tau ,k}(x_{1})|=O(h^{r+1}),  \label{sup_delta}
\end{equation}%
because $q_{k}(\tau |x_{1})-\gamma _{\tau ,k}(x_{1})^{\top }c(x_{1}-x)$ is a
residual from the Taylor expansion of $q_{k}(\tau |x_{1})$ and $\mathcal{X}$
is bounded, and the derivatives from the Taylor expansion are bounded
uniformly over $P\in \mathcal{P}$.\newline
Let $f_{\tau ,k,x}^{\Delta }(t|x^{\prime })$ be the conditional density of $%
\Delta _{x,\tau ,lk,i}$ given $X_{i}=x^{\prime }$. For all $x^{\prime }\in 
\mathbf{R}^{d}$ such that $|x-x^{\prime }|\leq h$,
\begin{eqnarray}
f_{\tau ,k,x}^{\Delta }(t|x^{\prime }) &=&(\partial /\partial t)P\left\{
B_{li}-q_{k}(\tau |X_{i})\leq t-\delta _{n,\tau ,k}(X_i;x)|X_{i}=x^{\prime }\right\}  \label{dec23} \\
&=&f_{\tau ,k}(t-\delta _{n,\tau ,k}(x^{\prime };x)|x^{\prime }).  \notag
\end{eqnarray}%
Since $f_{\tau ,k}(\cdot |x^{\prime })$ is bounded uniformly over $x^{\prime
}\in \mathcal{S}_{\tau }(\varepsilon )$ and over $\tau \in \mathcal{T}$
(Assumption QR1(ii)), we conclude that for some $C>0$ that does not depend
on $P\in \mathcal{P}$,%
\begin{equation}
\sup_{\tau \in \mathcal{T}}\sup_{x^{\prime },x\in \mathcal{S}_{\tau
}(\varepsilon )} \sup_{t \in \mathbf{R}} f_{\tau ,k,x}^{\Delta }(t|x^{\prime })<C.  \label{bd54}
\end{equation}%
We will use the results in (\ref{sup_delta})\ and (\ref{bd54}) later.\newline

Following the identity in Knight (1998, see the proof of Theorem 1), we write%
\begin{equation*}
l_{\tau }(x-y)-l_{\tau }(x)=-y\cdot \bar{l}_{\tau }(x)+\mu (x,y),
\end{equation*}%
where $\mu (x,y)\equiv y\int_{0}^{1}\{1\{x\leq ys\}-1\{x\leq 0\}\}ds$ and
\begin{equation*}
\bar{l}_{\tau }\left( x\right) \equiv \tau -1\{x\leq 0\}.
\end{equation*}%
Hence $\zeta _{n,x,\tau ,k}^{\Delta }(a,b)-\mathbf{E}[\zeta _{n,x,\tau
,k}^{\Delta }(a,b)]$ is equal to
\begin{equation*}
\sum_{i=1}^{n}\left\{ G_{n,x,\tau ,k}(S_{i};a,b)-\mathbf{E}\left[
G_{n,x,\tau ,k}(S_{i};a,b)\right] \right\},
\end{equation*}
where $S_{i}\equiv (B_{i}^{\top },X_{i}^{\top },L_{i})^{\top }$, $%
B_{i}=(B_{1,i},...,B_{\bar{L},i})^{\top }$, and%
\begin{equation}
G_{n,x,\tau }(S_{i};a,b)\equiv \int_{0}^{1}g_{n,x,\tau ,k}(S_{i};s,b,a)ds
\label{int}
\end{equation}%
and $g_{n,x,\tau ,k}(S_{i};s,b,a)$ is defined to be 
\begin{equation*}
1\left\{ L_{i}=k\right\} \sum_{l=1}^{k}\left( 
\begin{array}{c}
1\left\{ \Delta _{x,\tau ,lk,i}-a^{\top }c_{h,x,i}/\sqrt{nh^{d}}\leq \left(
sb\right) ^{\top }c_{h,x,i}/\sqrt{nh^{d}}\right\} \\ 
-1\left\{ \Delta _{x,\tau ,lk,i}-a^{\top }c_{h,x,i}/\sqrt{nh^{d}}\leq
0\right\}%
\end{array}%
\right) \frac{b^{\top }c_{h,x,i}K_{h,x,i}}{\sqrt{nh^{d}}}.
\end{equation*}%
Let $\mathcal{G}_{n}\equiv \{G_{n,x,\tau ,k}(\cdot ;a,b):(a,b,x)\in \lbrack
-\delta _{1n},\delta _{1n}]^{r+1}\times \lbrack -\delta _{2n},\delta
_{2n}]^{r+1}\times \mathcal{S}_{\tau }(\varepsilon ),\tau \in \mathcal{T}\}$,%
\begin{eqnarray*}
\mathcal{G}_{1n} &\equiv &\{\lambda _{\tau ,1n}(\cdot ;a,x):(a,x)\in \lbrack
-\delta _{1n},\delta _{1n}]^{r+1}\times \mathcal{S}_{\tau }(\varepsilon
),\tau \in \mathcal{T}\} \\
\mathcal{G}_{2n} &\equiv &\{b^{\top }\lambda _{\tau ,2n}(\cdot ;x):(b,x)\in
\lbrack -\delta _{2n},\delta _{2n}]^{r+1}\times \mathcal{S}_{\tau
}(\varepsilon ),\tau \in \mathcal{T}\}\text{ and} \\
\mathcal{G}_{3n} &\equiv &\{\lambda _{\tau ,3n}(\cdot ;x):x\in \mathcal{S}%
_{\tau }(\varepsilon ),\tau \in \mathcal{T}\},
\end{eqnarray*}%
where $\lambda _{\tau ,1n}(S_{i};a,x)\equiv (\Delta _{x,\tau ,lk,i}-a^{\top
}\lambda _{\tau ,2n}(S_{i};x))_{l=1}^{\bar{L}}$,%
\begin{equation*}
\lambda _{\tau ,2n}(S_{i};x)\equiv c_{h,x,i}/\sqrt{nh^{d}}\text{ and\ }%
\lambda _{\tau ,3n}(S_{i};x)\equiv K_{h,x,i}.
\end{equation*}

First, we compute the entropy bound for $\mathcal{G}_{n}.$ We focus on $%
\mathcal{G}_{1n}$ first. By Assumption QR1(iv), there exists $C>0$ that does
not depend on $P\in \mathcal{P}$, such that for any $\tau \in \mathcal{T}$,
any $(a,x)$ and$\ (a^{\prime },x^{\prime })$ in $[-\delta _{1n},\delta
_{1n}]^{r+1}\times \mathcal{S}_{\tau }(\varepsilon )$, and any $\tau ,\tau
^{\prime }\in \mathcal{T}$, 
\begin{equation*}
\left\vert \lambda _{\tau ,1n}(S_{i};a,x)-\lambda _{\tau ^{\prime
},1n}(S_{i};a^{\prime },x^{\prime })\right\vert \leq C n^s\left\{ ||a-a^{\prime }||+|\tau -\tau ^{\prime
}|+||x-x^{\prime }||\right\},
\end{equation*}
for some $s>0$. Observe that for any $\varepsilon ^{\prime }>0$,%
\begin{equation*}
N\left( \varepsilon ^{\prime },[-\delta _{1n},\delta _{1n}]^{r+1}\times 
\mathcal{S}_{\tau }(\varepsilon ),||\cdot ||\right) \leq (\delta
_{1n}/\varepsilon ^{\prime })^{r+2},
\end{equation*}%
because $\mathcal{S}_{\tau }(\varepsilon )$ is bounded in the Euclidean
space uniformly in $\tau \in \mathcal{T}$. Hence there are $C,s'>0$ such that
for all $\varepsilon ^{\prime }\in (0,1],$%
\begin{equation*}
\log N(\varepsilon ^{\prime },\mathcal{G}_{1n},||\cdot ||_{\infty })\leq
-C\log \left((\varepsilon ^{\prime }/\delta _{1n}) n^{-s'}\right),
\end{equation*}%
where $||\cdot ||_{\infty }$ denotes the usual supremum norm. Applying
similar arguments to $\mathcal{G}_{2n}$ and $\mathcal{G}_{3n}$, we conclude
that 
\begin{equation}
\log N(\varepsilon ^{\prime },\mathcal{G}_{mn},||\cdot ||_{\infty })\leq
C-C\log (\varepsilon ^{\prime }/n)\text{, }m=1,2,3,  \label{eb3}
\end{equation}%
for some $C>0$.

Define for $x\in \mathbf{R}$, $\delta >0,$%
\begin{eqnarray*}
1_{\delta }^{L}(x) &\equiv &\left( 1-\min \{x/\delta ,1\}\right)
1\{0<x\}+1\{x\leq 0\}\text{ and} \\
1_{\delta }^{U}(x) &\equiv &\left( 1-\min \{(x/\delta )+1,1\}\right)
1\{0<x+\delta \}+1\{x+\delta \leq 0\}.
\end{eqnarray*}%
We also define for $x,y,z\in \mathbf{R}$,%
\begin{eqnarray*}
\mu (x,y,z) &\equiv &zy\int_{0}^{1}\{1\{x\leq ys\}-1\{x\leq 0\}\}ds, \\
\mu _{\delta }^{U}(x,y,z) &\equiv &zy\int_{0}^{1}\{1\{x\leq ys\}-1_{\delta
}^{U}(x)\}ds,\text{ and} \\
\mu _{\delta }^{L}(x,y,z) &\equiv &zy\int_{0}^{1}\{1\{x\leq ys\}-1_{\delta
}^{L}(x)\}ds.
\end{eqnarray*}%
Then observe that
\begin{eqnarray}
\mu _{\delta }^{L}(x,y,z) &\leq &\mu (x,y,z)\leq \mu _{\delta }^{U}(x,y,z)
\label{ineq4} \\
\left\vert \mu _{\delta }^{U}(x,y,z)-\mu (x,y,z)\right\vert &\leq
&|zy|1\{|x|<\delta \}  \notag \\
\left\vert \mu _{\delta }^{L}(x,y,z)-\mu (x,y,z)\right\vert &\leq
&|zy|1\{|x|<\delta \}  \notag \\
\left\vert \mu _{\delta }^{U}(x,y,z)-\mu _{\delta }^{U}(x^{\prime
},y^{\prime },z^{\prime })\right\vert &\leq &C\{\left\vert y-y^{\prime
}\right\vert +|z-z^{\prime }|+|x-x^{\prime }|/\delta \},\text{ and}  \notag
\\
\left\vert \mu _{\delta }^{L}(x,y,z)-\mu _{\delta }^{L}(x^{\prime
},y^{\prime },z^{\prime })\right\vert &\leq &C\{\left\vert y-y^{\prime
}\right\vert +|z-z^{\prime }|+|x-x^{\prime }|/\delta \},  \notag
\end{eqnarray}%
for any $y,y^{\prime },x,x^{\prime },z,z^{\prime }\in \mathbf{R}$. Define%
\begin{eqnarray*}
\mathcal{G}_{n,\delta }^{U} &\equiv &\left\{ \mu _{\delta
}^{U}(g_{1}(S_{i}),g_{2}(S_{i}),g_{3}(S_{i})):g_{m}\in \mathcal{G}_{mn},\
m=1,2,3\right\} ,\text{ and} \\
\mathcal{G}_{n,\delta }^{L} &\equiv &\left\{ \mu _{\delta
}^{L}(g_{1}(S_{i}),g_{2}(S_{i}),g_{3}(S_{i})):g_{m}\in \mathcal{G}_{mn},\
m=1,2,3\right\} .
\end{eqnarray*}%
From (\ref{ineq4}) and (\ref{eb3}), we find that there exists $C>0$ such
that for each $\delta >0$ and $\varepsilon >0,$%
\begin{eqnarray}
\log N_{[]}(C\varepsilon ,\mathcal{G}_{n,\delta }^{U},L_{p}(P)) &\leq
&C-C\log (\varepsilon \delta /n)\text{ and}  \label{entropy bound} \\
\log N_{[]}(C\varepsilon ,\mathcal{G}_{n,\delta }^{L},L_{p}(P)) &\leq
&C-C\log (\varepsilon \delta /n).  \notag
\end{eqnarray}%
Fix $\varepsilon >0,$ set $\delta =\varepsilon ,$ and take brackets $%
[g_{1,L}^{(\varepsilon )},g_{1,U}^{(\varepsilon
)}],...,[g_{N,L}^{(\varepsilon )},g_{N,U}^{(\varepsilon )}]$ and $[\tilde{g}%
_{1,L}^{(\varepsilon )},\tilde{g}_{1,U}^{(\varepsilon )}],...,[\tilde{g}%
_{N,L}^{(\varepsilon )},\tilde{g}_{N,U}^{(\varepsilon )}]$ such that%
\begin{eqnarray}
\mathbf{E}\left( |g_{s,U}^{(\varepsilon )}(S_{i})-g_{s,L}^{(\varepsilon
)}(S_{i})|^{2}\right) &\leq &\varepsilon ^{2}\text{ and}  \label{bounds33} \\
\mathbf{E}\left( |\tilde{g}_{s,U}^{(\varepsilon )}(S_{i})-\tilde{g}%
_{s,L}^{(\varepsilon )}(S_{i})|^{2}\right) &\leq &\varepsilon ^{2},  \notag
\end{eqnarray}%
and for any $g\in \mathcal{G}_{n}^{U}$ and $\tilde{g}\in \mathcal{G}_{n}^{L}$%
, there exists $s\in \{1,...,N\}$ such that $g_{s,L}^{(\varepsilon )}\leq
g\leq g_{s,U}^{(\varepsilon )}$ and $\tilde{g}_{s,L}^{(\varepsilon )}\leq 
\tilde{g}\leq \tilde{g}_{s,U}^{(\varepsilon )}$. Without loss of generality,
we assume that $g_{s,L}^{(\varepsilon )},g_{s,U}^{(\varepsilon )}\in 
\mathcal{G}_{n}^{U}$ and $\tilde{g}_{s,L}^{(\varepsilon )},\tilde{g}%
_{s,U}^{(\varepsilon )}\in \mathcal{G}_{n}^{L}$. By the first inequality in (%
\ref{ineq4}), we find that the brackets\ $[\tilde{g}_{s,L}^{(\varepsilon
)},g_{s,U}^{(\varepsilon )}],\ k=1,...,N,$ cover $\mathcal{G}_{n}$. Hence by
putting $\delta =\varepsilon $ in (\ref{entropy bound}) and redefining
constants, we conclude that for some $C>0$%
\begin{equation}
\log N_{[]}(C\varepsilon ,\mathcal{G}_{n},L_{p}(P))\leq C-C\log (\varepsilon
/n),  \label{entropy bound2}
\end{equation}%
for all $\varepsilon >0$.

Now, observe that%
\begin{equation}
\sup_{b:||b||\leq \delta _{2n},\tau \in \mathcal{T},x\in \mathcal{S}_{\tau
}(\varepsilon )}\left\vert \frac{b^{\top }c_{h,x,i}K_{h,x,i}}{\sqrt{nh^{d}}}%
\right\vert \leq \frac{\bar{c}||K||_{\infty }\delta _{2n}}{\sqrt{nh^{d}}}%
\text{,}  \label{ineq51}
\end{equation}%
where $\bar{c}>0$ is the diameter of the compact support of $K$.

For any $g\in \mathcal{G}_{n}/\bar{L}$ and any $m\geq 1$, we bound
\begin{equation}
\label{bd344}
|g(S_{i})|^{m} \le \left\vert \frac{b^{\top }c_{h,x,i}K_{h,x,i}}{\sqrt{nh^{d}}}\right\vert ^{m}.
\end{equation}
Also, for any $g\in \mathcal{G}_{n}/\bar{L}$ , we use (\ref{int}) and bound $\mathbf{E}\left[|g(S_{i})|^2|X_{i},L_{i}=k\right]$ by
\begin{eqnarray*}
	\left\vert\frac{b^{\top }c_{h,x,i}K_{h,x,i}}{\sqrt{nh^{d}}} \right\vert^2 P\left\{-\left\vert\frac{b^{\top }c_{h,x,i}K_{h,x,i}}{\sqrt{nh^{d}}} \right\vert \le \Delta_{x,\tau,lk,i}-\frac{a^{\top }c_{h,x,i}K_{h,x,i}}{\sqrt{nh^{d}}} \le \left\vert\frac{b^{\top }c_{h,x,i}K_{h,x,i}}{\sqrt{nh^{d}}} \right\vert \vert X_i,L_i=k \right\}, 
\end{eqnarray*}
where $a \in [-\delta_{1n},\delta_{1n}]^{r+1}$ and $b \in [-\delta_{2n},\delta_{2n}]^{r+1}$ by the definition of $\mathcal{G}_n$. Using (\ref{ineq51}), we bound the last expression by
\begin{eqnarray*}
	C_1\frac{\delta_{2n}^3}{(nh^d)^{3/2}} \cdot \sup_{P\in 
		\mathcal{P}}P\left\{ \max_{s=1,...,d}|X_{is}-x|\leq h/2\right\}
	\le C_2\frac{\delta_{2n}^3 h^d}{(nh^d)^{3/2}},
\end{eqnarray*}
for some constants $C_1,C_2>0$. Therefore, by (\ref{ineq51}), for some constants $C_{1},C_{2}>0,$ it is
satisfied that for any $m\geq 2$,
\begin{equation*}
\sup_{P\in \mathcal{P}}\mathbf{E}\left[ |g(S_{i})|^{m}\right] \leq
C_{1}\left( \frac{\delta _{2n}}{\sqrt{nh^{d}}}\right) ^{m-2}\cdot \sup_{P\in 
\mathcal{P}} \mathbf{E}\left[|g(S_{i})|^2\right] \leq
C_{2}b_{n}^{m-2}s_{n}^{2},
\end{equation*}
where%
\begin{equation}
b_{n}\equiv \frac{\delta _{2n}}{\sqrt{nh^{d}}}\text{ and }s_{n}\equiv \frac{
\delta_{2n}^{3/2}}{n^{3/4}h^{d/4}}.  \label{b and s}
\end{equation}%
By (\ref{ineq4}), (\ref{bounds33}), and (\ref{ineq51}), and the definition
of $b_{n}$ and $s_{n}$ in (\ref{b and s}), there exist constants $%
C_{1},C_{2}>0$ such that for all $m\geq 2,$%
\begin{eqnarray*}
\mathbf{E}\left( |g_{s,U}^{(\varepsilon )}(S_{i})-\tilde{g}%
_{s,L}^{(\varepsilon )}(S_{i})|^{m}\right) &=&\mathbf{E}\left(
|g_{s,U}^{(\varepsilon )}(S_{i})-\tilde{g}_{s,L}^{(\varepsilon
)}(S_{i})|^{m-2}|g_{s,U}^{(\varepsilon )}(S_{i})-\tilde{g}%
_{s,L}^{(\varepsilon )}(S_{i})|^{2}\right) \\
&\leq &C_{1}\cdot b_{n}^{m-2}\cdot \mathbf{E}\left( |g_{s,U}^{(\varepsilon
)}(S_{i})-\tilde{g}_{s,L}^{(\varepsilon )}(S_{i})|^{2}\right) \\
&\leq &2C_{1}\cdot b_{n}^{m-2}\cdot \mathbf{E}\left( |g_{s,U}^{(\varepsilon
)}(S_{i})-g_{s,L}^{(\varepsilon )}(S_{i})|^{2}\right) \\
&&+2C_{1}\cdot b_{n}^{m-2}\cdot \mathbf{E}\left( |g_{s,L}^{(\varepsilon
)}(S_{i})-\tilde{g}_{s,L}^{(\varepsilon )}(S_{i})|^{2}\right) \\
&\leq &2C_{2}\cdot b_{n}^{m-2}\cdot \{\varepsilon ^{2}+b_{n}^{2}\varepsilon
\}\ \mathbf{\leq \ }2C_{2}\cdot b_{n}^{m-2}\cdot \varepsilon .
\end{eqnarray*}%
(The term $b_{n}^{2}\varepsilon $ is obtained by chaining the second and
third inequalities of (\ref{ineq4}) and using the fact that $\delta
=\varepsilon $ and the uniform bound in (\ref{bd54}). The last inequality
follows because $b_{n}\rightarrow 0$ as $n\rightarrow \infty $.) We define $%
\bar{\varepsilon}=\varepsilon ^{1/2}$ and bound the last term by $%
C_{3}b_{n}^{m-2}\bar{\varepsilon}^{2},$ for some $C_{3}>0$, because $%
b_{n}\leq 1$ from some large $n$ on. The entropy bound in (\ref{entropy
bound2}) as a function of $\bar{\varepsilon}$ remains the same except for a
different constant $C>0$ there.

Now by Theorem 6.8 of \citeasnoun{Massart:07} and (\ref{entropy bound2}),
there exist $C_{1},C_{2}>0$ such that%
\begin{eqnarray}
&&\sup_{P\in \mathcal{P}}\mathbf{E}\left[ \sup_{a,b:||a||\leq \delta
_{1n},||b||\leq \delta _{2n},\tau \in \mathcal{T},x\in \mathcal{S}_{\tau
}(\varepsilon )}|\zeta _{n,x,\tau ,k}^{\Delta }(a,b)-\mathbf{E}[\zeta
_{n,x,\tau ,k}^{\Delta }(a,b)]|\right]  \label{mi} \\
&\leq &C_{1}\sqrt{n}\int_{0}^{s_{n}}\sqrt{n\wedge \left\{ -\log \left( \frac{%
\varepsilon }{n}\right) \right\} }d\varepsilon +C_{1}(b_{n}+s_{n})\log n 
\notag \\
&\leq &C_{2}s_{n}\sqrt{n\log n}+C_{2}b_{n}\log n=O\left( \frac{\delta _{2n}^{3/2}
\sqrt{\log n}}{n^{1/4}h^{d/4}}\right) ,  \notag
\end{eqnarray}%
where the last equality follows by the definitions of $b_{n}$ and $s_{n}$ in
(\ref{b and s}) and by Assumption QR2(ii).\newline
\noindent (ii) Define $\lambda _{\tau ,4n}(S_{i};x)\equiv \Delta _{x,\tau
,lk,i}$ and $\mathcal{L}_{k,1}\equiv \{\tilde{l}_{\tau }(\lambda _{\tau
,4n}(\cdot ;x)):\tau \in \mathcal{T},x\in \mathcal{S}_{\tau }(\varepsilon
)\} $, and $\mathcal{L}_{k,2}\equiv \{\lambda _{\tau ,2n}(\cdot ;x)\lambda
_{\tau ,3n}(\cdot ;x):\tau \in \mathcal{T},x\in \mathcal{S}_{\tau
}(\varepsilon )\}$. We write%
\begin{equation*}
\psi _{n,x,\tau ,k}=\left\{ \psi _{n,x,\tau ,k}-\mathbf{E}\left[ \psi
_{n,x,\tau ,k}\right] \right\} +\mathbf{E}\left[ \psi _{n,x,\tau ,k}\right] .
\end{equation*}%
The leading term is an empirical process indexed by the functions in $%
\mathcal{L}_{k}\equiv \mathcal{L}_{k,1}\cdot \mathcal{L}_{k,2}$.
Approximating the indicator function in $\tilde{l}_{\tau }$ by upper and
lower Lipschitz functions and following similar arguments in the proof of
(i), we find that%
\begin{equation*}
\sup_{P\in \mathcal{P}}\log N_{[]}(\varepsilon ,\mathcal{L}%
_{k},L_{p}(P))\leq C-C\log \varepsilon +C\log n,
\end{equation*}%
for some constant $C>0$. Note that we can take a constant function $C$ as an
envelope of $\mathcal{L}_{k}$. Then we follow the proof of Lemma 2 to obtain
that%
\begin{equation*}
\mathbf{E}\left[ \sup_{\tau \in \mathcal{T},x\in \mathcal{S}_{\tau
}(\varepsilon )}\left\Vert \psi _{n,x,\tau ,k}-\mathbf{E}\left[ \psi
_{n,x,\tau ,k}\right] \right\Vert \right] =O(\sqrt{\log n})\text{, uniformly
in }P\in \mathcal{P}\text{.}
\end{equation*}%
By using (\ref{sup_delta})\ and (\ref{dec23}), we find that%
\begin{equation}
\label{ineq991}
\mathbf{E}\left[ \psi _{n,x,\tau ,k}\right] =O(h^{r+1})=o(\sqrt{\log n})%
\text{, uniformly in }P\in \mathcal{P,}
\end{equation}%
because $\sqrt{nh^{d}}h^{r+1}/\sqrt{\log n}\rightarrow 0$ by Assumption
QR2(ii).

(iii) Recall the definition of $g_{n,x,\tau ,k}(S_{i};s,b,a)$ in the proof
of Lemma QR1(i). We write 
\begin{equation*}
\mathbf{E}[\zeta _{n,x,\tau ,k}^{\Delta }(a,b)]=n\int_{0}^{1}\mathbf{E}\left[
g_{n,x,\tau ,k}(S_{i};s,b,a)\right] ds.
\end{equation*}%
Using change of variables, we rewrite%
\begin{equation*}
\int_{0}^{1}\mathbf{E}\left[ g_{n,x,\tau ,k}(S_{i};s,b,a)\right]
ds=kP\left\{ L_{i}=k|X_{i}\right\} \cdot \phi _{n}(X_{i};a,b)\text{,}
\end{equation*}%
where 
\begin{equation*}
\phi _{n}(X_{i};a,b)=\int_{a^{\top }c_{h,x,i}/\sqrt{nh^{d}}}^{(b+a)^{\top
}c_{h,x,i}/\sqrt{nh^{d}}}\left\{ 
\begin{array}{c}
F_{\tau ,k}\left( u-\delta _{n,\tau ,k}(X_{i};x)|X_{i}\right) \\ 
-F_{\tau ,k}\left( -\delta _{n,\tau ,k}(X_{i};x)|X_{i}\right)%
\end{array}%
\right\} du\cdot K_{h,x,i}.
\end{equation*}%
By expanding the difference, we have 
\begin{equation*}
\phi _{n}(X_{i};a,b)=\int_{a^{\top }c_{h,x,i}/\sqrt{nh^{d}}}^{(b+a)^{\top
}c_{h,x,i}/\sqrt{nh^{d}}}udu\cdot f_{\tau ,k}\left( -\delta _{n,\tau
,k}(X_{i};x)|X_{i}\right) \cdot K_{h,x,i}+R_{n,x,i}(a,b),
\end{equation*}%
where $R_{n,x,i}(a,b)$ denotes the remainder term in the expansion. As for
the leading integral,%
\begin{equation*}
\int_{a^{\top }c_{h,x,i}/\sqrt{nh^{d}}}^{(b+a)^{\top }c_{h,x,i}/\sqrt{nh^{d}}%
}udu=\frac{1}{2nh^{d}}\left\{ b^{\top }c_{h,x,i}c_{h,x,i}^{\top
}(b+2a)\right\} .
\end{equation*}%
Hence, for any sequences $a_{n},b_{n},$ we can write $\mathbf{E}[\zeta
_{n,x,\tau ,k}^{\Delta }(a_{n},b_{n})]$ as%
\begin{eqnarray*}
&&\frac{1}{2}b_{n}^{\top }h^{-d} k \mathbf{E}\left[ P\left\{
L_{i}=k|X_{i}\right\} f_{\tau ,k}\left( -\delta _{n,\tau
,k}(X_{i};x)|X_{i}\right) c_{h,x,i}c_{h,x,i}^{\top }\cdot K_{h,x,i}\right]
(b_{n}+2a_{n}) \\
&&+nk\mathbf{E}\left[ P\left\{ L_{i}=k|X_{i}\right\} R_{n,x,i}(a_{n},b_{n})%
\right] \\
&=&\frac{1}{2}b_{n}^{\top }M_{n,\tau ,k}(x)(b_{n}+2a_{n})+nk\mathbf{E}\left[
P\left\{ L_{i}=k|X_{i}\right\} R_{n,x,i}(a_{n},b_{n})\right].
\end{eqnarray*}%
We can bound%
\begin{eqnarray*}
&&nk\left\vert \mathbf{E}\left[ P\left\{ L_{i}=k|X_{i}\right\}
R_{n,x,i}(a_{n},b_{n})\right] \right\vert \\
&\leq &C_{1}nk\mathbf{E}\left[ \int_{a_{n}^{\top }c_{h,x,i}/\sqrt{nh^{d}}%
}^{(b_{n}+a_{n})^{\top }c_{h,x,i}/\sqrt{nh^{d}}}u^{2}du\cdot K_{h,x,i}\right]
\leq \frac{C_{2}b_{n}a_{n}^{2}}{n^{1/2}h^{d/2}},
\end{eqnarray*}%
where $C_{1}>0$ and $C_{2}>0$ are constants that do not depend on $n$ or $%
P\in \mathcal{P}$.
\end{proof}

\begin{proof}[Proof of Theorem 1]
(i) Let 
\begin{eqnarray}
\tilde{u}_{n,x,\tau } &\equiv &-M_{n,\tau ,k}^{-1}(x)\psi _{n,x,\tau ,k},\ 
\label{ineq5} \\
\tilde{\psi}_{n,x,\tau ,k}(b) &\equiv &b^{\top }\psi _{n,x,\tau ,k}+b^{\top
}M_{n,\tau ,k}(x)b/2\text{, and}  \notag \\
\tilde{\psi}_{n,x,\tau ,k}(a,b) &\equiv &\tilde{\psi}_{n,x,\tau ,k}(a+b)-%
\tilde{\psi}_{n,x,\tau ,k}(a).  \notag
\end{eqnarray}%
For any $a\in \mathbf{R}^{|A_{r}|},$ we can write%
\begin{eqnarray}
\tilde{\psi}_{n,x,\tau ,k}(\tilde{u}_{n,x,\tau },a-\tilde{u}_{n,x,\tau }) &=&%
\tilde{\psi}_{n,x,\tau ,k}(a)-\tilde{\psi}_{n,x,\tau ,k}(\tilde{u}_{n,x,\tau
})  \label{bound} \\
&=&\left( a-\tilde{u}_{n,x,\tau }\right) ^{\top }M_{n,\tau ,k}(x)\left( a-%
\tilde{u}_{n,x,\tau }\right) /2  \notag \\
&\geq &C_{1}||a-\tilde{u}_{n,x,\tau }||^{2},  \notag
\end{eqnarray}%
where $C_{1}>0$ is a constant that does not depend on $\tau \in \mathcal{T}%
,\ x\in \mathcal{S}_{\tau }(\varepsilon )$ or $P\in \mathcal{P}$. The last
inequality uses Assumption QR1 and the fact that $K$ is a nonnegative map
that is not constant at zero and Lipschitz continuous.

Let 
\begin{equation*}
\hat{u}_{n,x,\tau }\equiv \sqrt{nh^{d}}H(\hat{\gamma}_{\tau ,k}(x)-\gamma
_{\tau ,k}(x)),
\end{equation*}%
where $x\in \mathcal{S}_{\tau }(\varepsilon )$ and $\tau \in \mathcal{T}$.
Since $\zeta _{n,x,\tau ,k}(\tilde{u}_{n,x,\tau },b)$ is convex in $b,$ we
have for any $0<\delta \leq l$ and for any $b\in \mathbf{R}^{|A_{r}|}$ such
that $||b||=1,$%
\begin{eqnarray}
(\delta /l)\zeta _{n,x,\tau ,k}(\tilde{u}_{n,x,\tau },lb) &\geq &\zeta
_{n,x,\tau ,k}(\tilde{u}_{n,x,\tau },\delta b)  \label{bound2} \\
&\geq &\tilde{\psi}_{n,x,\tau ,k}(\tilde{u}_{n,x,\tau },\delta b)-\Delta
_{n,k}(\delta ),  \notag
\end{eqnarray}%
where 
\begin{equation*}
\Delta _{n,k}(\delta )\equiv \sup_{b\in \mathbf{R}^{|A_{r}|}:||b||\leq
1}|\zeta _{n,x,\tau ,k}(\tilde{u}_{n,x,\tau },\delta b)-\tilde{\psi}%
_{n,x,\tau ,k}(\tilde{u}_{n,x,\tau },\delta b)|.
\end{equation*}%
Therefore, if $||\hat{u}_{n,x,\tau }-\tilde{u}_{n,x,\tau }||\geq \delta $,
we replace $b$ by $\hat{u}_{n,x,\tau }^{\Delta }=(\hat{u}_{n,x,\tau }-\tilde{%
u}_{n,x,\tau })/||\hat{u}_{n,x,\tau }-\tilde{u}_{n,x,\tau }||$ and $l$ by $||%
\hat{u}_{n,x,\tau }-\tilde{u}_{n,x,\tau }||$ in (\ref{ineq5}), and use (\ref%
{bound2}) to obtain that%
\begin{eqnarray}
0 &\geq &\zeta _{n,x,\tau ,k}(\tilde{u}_{n,x,\tau },||\hat{u}_{n,x,\tau }-%
\tilde{u}_{n,x,\tau }||\hat{u}_{n,x,\tau }^{\Delta })  \label{ineq33} \\
&\geq &\zeta _{n,x,\tau ,k}(\tilde{u}_{n,x,\tau },\delta \hat{u}_{n,x,\tau
}^{\Delta })  \notag \\
&\geq &\tilde{\psi}_{n,x,\tau ,k}(\tilde{u}_{n,x,\tau },\delta \hat{u}%
_{n,x,\tau }^{\Delta })-\Delta _{n,k}(\delta )  \notag \\
&\geq &C_{1}\delta ^{2}||\hat{u}_{n,x,\tau }^{\Delta }||^{2}-\Delta
_{n,k}(\delta )=C_{1}\delta ^{2}-\Delta _{n,k}(\delta ),  \notag
\end{eqnarray}%
for all $\delta \leq ||\hat{u}_{n,x,\tau }-\tilde{u}_{n,x,\tau }||$, where
the first inequality follows because $\zeta _{n,x,\tau ,k}(\tilde{u}%
_{n,x,\tau },||\hat{u}_{n,x,\tau }-\tilde{u}_{n,x,\tau }||b)$ is minimized
at $b=\hat{u}_{n,x,\tau }^{\Delta }$ by the definition of local polynomial
estimation, the second and the third inequality follows by (\ref{bound2}),
and the fourth inequality follows from (\ref{bound}), and the last equality
follows because $||\hat{u}_{n,x,\tau }^{\Delta }||^{2}=1.$

We take large $M>0$ and let%
\begin{equation}
\delta _{1n}=M\sqrt{\log n}\text{ and }\delta _{2n}=\frac{M\sqrt{\log n}}{%
n^{1/4}h^{d/4}}.  \label{cond4}
\end{equation}%
If $\delta _{2n}\leq ||\hat{u}_{n,x,\tau }-\tilde{u}_{n,x,\tau }||$, we have%
\begin{equation*}
C_{1}\delta _{2n}^{2}\leq \Delta _{n,k}(\delta _{2n}),
\end{equation*}%
from (\ref{ineq33}). We let 
\begin{equation*}
1_{n}\equiv 1\left\{ \sup_{\tau \in \mathcal{T},x\in \mathcal{S}_{\tau
}(\varepsilon )}||\tilde{u}_{n,x,\tau }||\leq M\delta _{1n}\right\} .
\end{equation*}%
Then we write
\begin{equation}
\label{ineq31}
P\left\{ \inf_{\tau \in \mathcal{T},x\in \mathcal{S}_{\tau }(\varepsilon )}||%
\hat{u}_{n,x,\tau }-\tilde{u}_{n,x,\tau }||^{2}\geq \delta _{2n}^{2}\right\}
\leq P\left\{ \Delta _{n,k}(\delta _{2n})1_{n}\geq \delta _{2n}^{2}\right\} +%
\mathbf{E}\left[ 1-1_{n}\right] .
\end{equation}
Now, we show that the first probability vanishes as $n\rightarrow \infty $.
For each $b\in \mathbf{R}^{|A_{r}|}$, using the definition of $\tilde{\psi}%
_{n,x,\tau ,k}(\tilde{u}_{n,x,\tau },b)=\tilde{\psi}_{n,x,\tau ,k}(\tilde{u}%
_{n,x,\tau }+b)-\tilde{\psi}_{n,x,\tau ,k}(\tilde{u}_{n,x,\tau })$, we write%
\begin{eqnarray*}
\tilde{\psi}_{n,x,\tau ,k}(\tilde{u}_{n,x,\tau },b) &=&\tilde{\psi}%
_{n,x,\tau ,k}(\tilde{u}_{n,x,\tau }+b)-\tilde{\psi}_{n,x,\tau ,k}(\tilde{u}%
_{n,x,\tau }) \\
&=&b^{\top }\psi _{n,x,\tau ,k}+(\tilde{u}_{n,x,\tau }+b)^{\top }M_{n,x,\tau
}(\tilde{u}_{n,x,\tau }+b)/2-\tilde{u}_{n,x,\tau }^{\top }M_{n,x,\tau }%
\tilde{u}_{n,x,\tau }/2 \\
&=&b^{\top }\psi _{n,x,\tau ,k}+b^{\top }M_{n,x,\tau }b/2+b^{\top
}M_{n,x,\tau }\tilde{u}_{n,x,\tau } \\
&=&b^{\top }M_{n,x,\tau }b/2.
\end{eqnarray*}%
Therefore,%
\begin{eqnarray*}
\zeta _{n,x,\tau ,k}(\tilde{u}_{n,x,\tau },b)-\tilde{\psi}_{n,x,\tau ,k}(%
\tilde{u}_{n,x,\tau },b) &=&\zeta _{n,x,\tau ,k}^{\Delta }(\tilde{u}%
_{n,x,\tau },b)-\mathbf{E}\left[ \zeta _{n,x,\tau ,k}^{\Delta }(\tilde{u}%
_{n,x,\tau },b)\right] \\
&&+\mathbf{E}\left[ \zeta _{n,x,\tau ,k}^{\Delta }(\tilde{u}_{n,x,\tau },b)%
\right] -b^{\top }M_{n,x,\tau }b/2+b^{\top }\psi _{n,x,\tau ,k} \\
&=&\zeta _{n,x,\tau ,k}^{\Delta }(\tilde{u}_{n,x,\tau },b)-\mathbf{E}\left[
\zeta _{n,x,\tau ,k}^{\Delta }(\tilde{u}_{n,x,\tau },b)\right] \\
&&+\mathbf{E}\left[ \zeta _{n,x,\tau ,k}^{\Delta }(\tilde{u}_{n,x,\tau },b)%
\right] -b^{\top }M_{n,x,\tau }(b+2\tilde{u}_{n,x,\tau })/2.
\end{eqnarray*}%
By Lemma QR1(i),%
\begin{eqnarray*}
&&\sup_{\tau \in \mathcal{T},x\in \mathcal{S}_{\tau }(\varepsilon
)}\sup_{b\in \mathbf{R}^{|A_{r}|}:||b||\leq \delta _{2n}}\left\vert \zeta
_{n,x,\tau ,k}^{\Delta }(\tilde{u}_{n,x,\tau },b)-\mathbf{E}\left[ \zeta
_{n,x,\tau ,k}^{\Delta }(\tilde{u}_{n,x,\tau },b)\right] \right\vert \\
&=& O_{P}\left( \frac{\delta _{2n}^{3/2}\sqrt{%
\log n}}{n^{1/4}h^{d/4}}\right) ,
\end{eqnarray*}%
by the definition in (\ref{cond4}). And by Lemma QR1(iii),%
\begin{eqnarray*}
&&\sup_{\tau \in \mathcal{T},x\in \mathcal{S}_{\tau }(\varepsilon
)}\sup_{b\in \mathbf{R}^{|A_{r}|}:||b||\leq \delta _{2n}}\left\vert \mathbf{E%
}\left[ \zeta _{n,x,\tau ,k}^{\Delta }(\tilde{u}_{n,x,\tau },b)\right]
-b^{\top }M_{n,x,\tau }(b+2\tilde{u}_{n,x,\tau })/2\right\vert \\
&=& O\left( \frac{\delta _{2n}\log n}{n^{1/2}h^{d/2}}\right) ,
\end{eqnarray*}%
by the definition in (\ref{cond4}) and Assumption QR2(ii). Thus we conclude
that%
\begin{equation}
|\Delta _{n,k}(\delta _{2n})|=O_{P}\left( \frac{\delta _{2n}^{3/2}\sqrt{\log n}}{%
n^{1/4}h^{d/4}}\right) ,  \label{rate2}
\end{equation}%
where the last $O_{P}$ term is uniform over $P\in \mathcal{P}$. We deduce
from (\ref{rate2}) that%
\begin{equation*}
\sup_{P\in \mathcal{P}}P\left\{ \Delta _{n,k}(\delta _{2n})1\left\{
\sup_{\tau \in \mathcal{T},x\in \mathcal{S}_{\tau }(\varepsilon )}||\tilde{u}%
_{n,x,\tau }||\leq \delta _{1n}\right\} \geq \delta _{2n}^{5/2}\right\}
\rightarrow 0\text{ as }n\rightarrow \infty
\end{equation*}%
and as $M\uparrow \infty .$ The proof is completed because 
\begin{equation*}
\sup_{P\in \mathcal{P}}P\left\{ \sup_{\tau \in \mathcal{T},x\in \mathcal{S}%
_{\tau }(\varepsilon )}||\tilde{u}_{n,x,\tau }||>\delta _{1n}\right\}
\rightarrow 0,
\end{equation*}%
as $n\rightarrow \infty $ and as $M\uparrow \infty $ by Lemma QR1(ii). Thus,
we conclude from (\ref{ineq31}) that%
\begin{equation*}
||\hat{u}_{n,x,\tau }-\tilde{u}_{n,x,\tau }||=O_{P}\left( \frac{\sqrt{\log n}%
}{n^{1/4}h^{d/4}}\right) ,\ \mathcal{P}\text{-uniformly.}
\end{equation*}
Now the desired result of Theorem 1 follows from the fact that
\begin{eqnarray*}
	M_{n,x,\tau}^{-1}\mathbf{E}\psi_{n,x,\tau,k} = O(h^{r+1})=o\left(\frac{\log^{1/2}n}{n^{1/4}h^{d/4}}\right),
\end{eqnarray*}
which follows by Assumption QR1, (\ref{ineq991}), and Assumption QR2.
\end{proof}

As mentioned in the main text, the convergence rate in the asymptotic linear
representation is slightly faster than the rate in Theorem 2 of %
\citeasnoun{Guerre/Sabbah:12}. To see this difference closely, %
\citeasnoun{Guerre/Sabbah:12} on page 118 wrote, for fixed numbers $x$ and $%
y,$ 
\begin{equation*}
l_{\tau }(\varepsilon _{i}+x+y)-l_{\tau }(\varepsilon _{i}+x)-y\cdot \tilde{l%
}_{\tau }(\varepsilon _{i}+x)=\int_{x}^{x+y}(1\{\varepsilon _{i}\leq
t\}-1\{\varepsilon _{i}\leq 0\}dt,
\end{equation*}%
where $\varepsilon _{i}$ is a certain random variable with density function,
say, $f$ which satisfies $||f||_{\infty }<\infty $. From this, %
\citeasnoun{Guerre/Sabbah:12} proceeded as follows:%
\begin{eqnarray*}
&&\mathbf{E}\left[ \left( l_{\tau }(\varepsilon _{i}+x+y)-l_{\tau
}(\varepsilon _{i}+x)-y\cdot \tilde{l}_{\tau }(\varepsilon _{i}+x)\right)
^{2}\right] \\
&\leq &2|y|\int_{x}^{x+y}\mathbf{E}\left[ (1\{\varepsilon _{i}\leq
t\}-1\{\varepsilon _{i}\leq 0\})^{2}\right] dt \\
&\leq &2|y|||f||_{\infty }\int_{x}^{x+y}|t|dt\leq
2|y|^{2}(|x|+|y|)||f||_{\infty }.
\end{eqnarray*}%
On the other hand, this paper considers \citeasnoun{Knight:98}'s inequality
and proceeds as follows:%
\begin{eqnarray*}
&&\mathbf{E}\left[ \left( l_{\tau }(\varepsilon _{i}+x+y)-l_{\tau
}(\varepsilon _{i}+x)-y\cdot \tilde{l}_{\tau }(\varepsilon _{i}+x)\right)
^{2}\right] \\
&\leq &|y|^{2}\int_{0}^{1}\mathbf{E}\left\vert 1\{\varepsilon _{i}+x\leq
yt\}-1\{\varepsilon _{i}+x\leq 0\}\right\vert dt \\
&\leq &|y|^{2}P\{-|y|\leq \varepsilon _{i}+x\leq |y|\}\leq 2 ||f||_{\infty
}|y|^{3}.
\end{eqnarray*}%
Note that when $|y|$ is decreasing to zero faster than $|x|$, the latter
bound is an improved one. The tighter $L^{2}$ bound gives a sharper bound
when we apply the maximal inequality of \citeasnoun{Massart:07} which yields
a slightly faster error rate. (Compare Proposition A.1 of %
\citeasnoun{Guerre/Sabbah:12} with Lemma QR1 where $\delta _{1n}$ and $\delta _{2n}$ in Lemma QR1 correspond to $t_\beta$ and $t_\epsilon$ in Proposition A.1 respectively.)

\begin{proof}[Proof of Corollary 1]
First, we write%
\begin{equation*}
M_{n,\tau ,k}^{-1}(x)\psi _{n,x,\tau ,k}=M_{n,\tau ,k}^{-1}(x)\left( \tilde{\psi}_{n,x,\tau ,k}+\frac{1}{\sqrt{%
		nh^{d}}}\sum_{i=1}^{n}\left( \rho _{n,i}(x,\tau )-\mathbf{E}\left[ \rho
_{n,i}(x,\tau )\right] \right) \right),
\end{equation*}%
where
\begin{equation*}
\rho _{n,i}(x,\tau )=1\left\{ L_{i}=k\right\} \sum_{l=1}^{L_{i}}\left( 
\tilde{l}_{\tau }\left( \Delta _{x,\tau ,lk,i}\right) -\tilde{l}_{\tau
}\left( \varepsilon _{\tau ,lk,i}\right) \right) c_{h,x,i}K_{h,x,i}.
\end{equation*}%
It suffices for Corollary 1 to show that
\begin{eqnarray*}
M_{n,\tau ,k}^{-1}(x)\frac{1}{\sqrt{nh^{d}}}\sum_{i=1}^{n}\left( \rho
_{n,i}(x,\tau )-\mathbf{E}\left[ \rho _{n,i}(x,\tau )\right] \right)
= O_{P}\left( \frac{\log ^{1/2}n}{n^{1/4}h^{d/4}}\right).
\end{eqnarray*}

Using the definition in (\ref{delta_tilde}), writing $\tilde{\delta}_{x,\tau
,k,i}=\tilde{\delta}_{\tau ,k}(X_{i};x),$ and using Knight's identity, we
write%
\begin{eqnarray*}
\rho _{n,i}(x,\tau ) &=&1\left\{ L_{i}=k\right\} \sum_{l=1}^{L_{i}}\left( 
\tilde{l}_{\tau }\left( \varepsilon _{\tau ,lk,i}+\tilde{\delta}_{\tau
,k}(X_{i};x)\right) -\tilde{l}_{\tau }\left( \varepsilon _{\tau
,lk,i}\right) \right) c_{h,x,i}K_{h,x,i} \\
&=&1\left\{ L_{i}=k\right\} \tilde{\delta}_{x,\tau
,k,i}\sum_{l=1}^{L_{i}}\left( \int_{0}^{1}\left( 
\begin{array}{c}
1\left\{ \varepsilon _{\tau ,lk,i}\leq -\tilde{\delta}_{x,\tau ,k,i}s\right\}
\\ 
-1\left\{ \varepsilon _{\tau ,lk,i}\leq 0\right\}%
\end{array}%
\right) ds\right) c_{h,x,i}K_{h,x,i}.
\end{eqnarray*}%
Following the same arguments in the proof of Lemma QR1(i), we deduce that 
\begin{equation*}
M_{n,\tau ,k}^{-1}(x)\frac{1}{\sqrt{nh^{d}}}\sum_{i=1}^{n}\left( \rho
_{n,i}(x,\tau )-\mathbf{E}\left[ \rho _{n,i}(x,\tau )\right] \right)
=O_{P}\left( \frac{\log ^{1/2}n}{n^{1/4}h^{d/4}}\right) ,
\end{equation*}%
uniformly over $\tau \in \mathcal{T},\ x\in \mathcal{S}_{\tau }(\varepsilon
) $ and over $P\in \mathcal{P.}$
\end{proof}

For $z=(x,\tau )\in \mathcal{Z}$ and $a,b\in \mathbf{R}$, we define 
\begin{equation*}
\zeta _{n,x,\tau ,k}^{\ast }(a,b)\equiv
\sum_{i=1}^{n}1\{L_{i}=k\}\sum_{l=1}^{k}\left\{ 
\begin{array}{c}
l_{\tau }\left( \Delta _{x,\tau ,lk,i}^{\ast }-(a+b)^{\top }c_{h,x,i}^{\ast
}/\sqrt{nh^{d}}\right) \\ 
-l_{\tau }\left( \Delta _{x,\tau ,lk,i}^{\ast }-b^{\top }c_{h,x,i}^{\ast }/%
\sqrt{nh^{d}}\right)%
\end{array}%
\right\} K_{h,x,i}^*.
\end{equation*}%
We also define%
\begin{equation*}
\zeta _{n,x,\tau ,k}^{\Delta \ast }(a,b)\equiv \zeta _{n,x,\tau ,k}^{\ast
}(a,b)-b^{\top }\psi _{n,x,\tau ,k}^{\ast }.
\end{equation*}%
The following lemma is the bootstrap analogue of Lemma QR1.

\begin{LemmaQR}
\textit{Suppose that Assumptions QR1-QR2 hold. Let }$\{\delta
_{1n}\}_{n=1}^{\infty }$\textit{\ and }$\{\delta _{2n}\}_{n=1}^{\infty }$%
\textit{\ be positive sequences such that }$\delta _{1n}=M\sqrt{\log n}$ 
\textit{for some }$M>0$ \textit{and }$\delta _{2n}\leq \delta _{1n}$\textit{%
\ from some large }$n$\textit{\ on}. \textit{Then for each }$k\in \mathbb{N}%
_{L},$ \textit{the following holds uniformly over} $P\in \mathcal{P}$:

\noindent (i)%
\begin{eqnarray*}
&&\mathbf{E}^{\ast }\left[ \sup_{a,b:||a||\leq \delta _{1n},||b||\leq \delta
_{2n}}\sup_{\tau \in \mathcal{T},x\in \mathcal{S}_{\tau }(\varepsilon
)}|\zeta _{n,x,\tau ,k}^{\Delta \ast }(a,b)-\mathbf{E}^{\ast }[\zeta
_{n,x,\tau ,k}^{\Delta \ast }(a,b)]|\right] \\
&=&O_{P}\left( \frac{\delta _{2n}^{3/2}\sqrt{\log n}}{n^{1/4}h^{d/4}}\right) .
\end{eqnarray*}%
\noindent (ii)%
\begin{eqnarray*}
&&\sup_{a,b:||a||\leq \delta _{1n},||b||\leq \delta _{2n}}\sup_{\tau \in 
\mathcal{T},x\in \mathcal{S}_{\tau }(\varepsilon )}\left\vert \mathbf{E}%
^{\ast }[\zeta _{n,x,\tau ,k}^{\Delta \ast }(a,b)]-\frac{b^{\top }M_{n,\tau
,k}(x)(b+2a)}{2}\right\vert \\
&=&O_{P}\left( \frac{\delta _{2n}^{3/2}\sqrt{\log n}}{n^{1/4}h^{d/4}}\right) .
\end{eqnarray*}
\end{LemmaQR}
Note that the convergence rate in Lemma QR1(ii) is slower than that in Lemma QR1(iii).

\begin{proof}[Proof of Lemma QR2]
(i) Similarly as in the proof of Lemma QR1(i), we rewrite $\zeta _{n,x,\tau
,k}^{\Delta \ast }(a,b)-\mathbf{E}^{\ast }[\zeta _{n,x,\tau ,k}^{\Delta \ast
}(a,b)]$ as%
\begin{equation*}
\sum_{i=1}^{n}\left\{ G_{n,x,\tau ,k}(S_{i}^{\ast };a,b)-\mathbf{E}\left[
G_{n,x,\tau ,k}(S_{i}^{\ast };a,b)\right] \right\} ,
\end{equation*}%
where $S_{i}^{\ast }=(Y_{i}^{\ast \top },X_{i}^{\ast \top })^{\top }$. Let $%
\pi =(x,\tau ,s,a,b)$ and $\Pi _{n}=\mathcal{S}(\varepsilon )\times \mathcal{%
T}\times \lbrack 0,1]\times \lbrack -\delta _{1n},\delta _{1n}]^{r+1}\times
\lbrack -\delta _{2n},\delta _{2n}]^{r+1}$, where $\mathcal{S}(\varepsilon
)=\{(x,\tau )\in \mathcal{X}\times \mathcal{T}:x\in \mathcal{S}_{\tau
}(\varepsilon )\}$. Using Proposition 2.5 of \citeasnoun{Gine:97},%
\begin{eqnarray*}
&&\mathbf{E}\left[ \mathbf{E}^{\ast }\left[ \sup_{a,b:||a||\leq \delta
_{1n},||b||\leq \delta _{2n}}\sup_{\tau \in \mathcal{T},x\in \mathcal{S}%
_{\tau }(\varepsilon )}|\zeta _{n,x,\tau ,k}^{\Delta \ast }(a,b)-\mathbf{E}%
^{\ast }[\zeta _{n,x,\tau ,k}^{\Delta \ast }(a,b)]|\right] \right] \\
&\leq &C\mathbf{E}\left[ \mathbf{E}_{N_{i}}\left( \sup_{\pi \in \Pi
_{n}}\left\vert \sum_{i=1}^{n}(N_{i}-1)\left\{ g_{n,x,\tau ,k}(S_{i};s,b,a)-%
\frac{1}{n}\sum_{i=1}^{n}g_{n,x,\tau ,k}(S_{i};s,b,a)\right\} \right\vert
\right) \right] ,
\end{eqnarray*}%
where $\{N_{i}\}_{i=1}^{n}$ are i.i.d. Poisson random variables with mean $1$
independent of $\{(Y_{i}^{\top },X_{i}^{\top })^{\top }\}_{i=1}^{\infty }$, $%
\mathbf{E}_{N_{i}}$ denotes expectation only with respect to the
distribution of $\{N_{i}\}_{i=1}^{n}$, and $g_{n,x,\tau ,k}(\cdot ;s,b,a)$ is as
defined in the proof of Lemma QR1(i). Here the constant $C>0$ does not
depend on $P\in \mathcal{P}$. We can bound the above by 
\begin{eqnarray*}
&&C\mathbf{E}\left[ \sup_{\pi \in \Pi _{n}}\left\vert
\sum_{i=1}^{n}(N_{i}-1)\left( g_{n,x,\tau ,k}(S_{i};s,b,a)-\mathbf{E}\left[
g_{n,x,\tau ,k}(S_{i};s,b,a)\right] \right) \right\vert \right] \\
&&+C\mathbf{E}\left( \left\vert \sum_{i=1}^{n}(N_{i}-1)\right\vert \right)
\times \mathbf{E}\left( \sup_{\pi \in \Pi _{n}}\left\vert \frac{1}{n}%
\sum_{i=1}^{n}g_{n,x,\tau ,k}(S_{i};s,b,a)-\mathbf{E}\left[ g_{n,x,\tau
,k}(S_{i};s,b,a)\right] \right\vert \right) .
\end{eqnarray*}%
The leading expectation is bounded by $O(\delta _{2n}\sqrt{\log n}%
/(n^{1/4}h^{d/4}))$ similarly as in the proof of Lemma QR1(i). And the
product of the two expectations in the second term is bounded by%
\begin{eqnarray*}
&&O(\sqrt{n})\times \frac{1}{n}\mathbf{E}\left( \sup_{\pi \in \Pi _{n}}\left\vert \sum_{i=1}^{n}\left\{ g_{n,x,\tau ,k}(S_{i};s,b,a)-\mathbf{E}\left[
g_{n,x,\tau ,k}(S_{i};s,b,a)\right] \right\} \right\vert \right) \\
&=&O\left( \delta _{2n}^{3/2}\sqrt{\log n}/(n^{3/4}h^{d/4})\right) ,
\end{eqnarray*}%
where the constant $C>0$ does not depend on $P\in \mathcal{P}$, and the last
equality follows similarly as in the proof of Lemma QR1(i).

\noindent (ii) Note that
\begin{equation}
\mathbf{E}^{\ast }[\zeta _{n,x,\tau ,k}^{\Delta \ast }(a,b)]=\mathbf{E}%
^{\ast }[\zeta _{n,x,\tau ,k}^{\Delta \ast }(a,b)]-\mathbf{E}[\zeta
_{n,x,\tau ,k}^{\Delta }(a,b)]+\mathbf{E}[\zeta _{n,x,\tau ,k}^{\Delta
}(a,b)].  \label{dec5}
\end{equation}%
The difference between the first two terms on the right hand side is 
\begin{equation*}
O_{P}\left( \frac{\delta _{2n}^{3/2}\sqrt{\log n}}{n^{1/4}h^{d/4}}\right),
\end{equation*}%
uniformly in $P\in \mathcal{P}$, as we have seen in (i). We apply Lemma
QR1(iii) to the last expectation in (\ref{dec5}) to obtain the desired
result.
\end{proof}

\begin{proof}[Proof of Theorem 2]
The proof is completed by using Lemma QR2 precisely in the same way as the
proof of Theorem 1 used Lemma QR1. While the convergence rate in Lemma QR2(ii) is slower than that in Lemma QR1(iii), we obtain the same convergence rate in the bootstrap version of (\ref{rate2}). Details are omitted.
\end{proof}

\begin{LemmaMIQ}
(i) \textit{Suppose that the conditions of Theorem 3(i) hold. Then
Asumptions A1-A3, A5-A6, and B1-B4 in LSW hold with the following
definitions:} $J=1,\ r_{n}\equiv \sqrt{nh^{3}},$%
\begin{eqnarray*}
v_{n,\tau }(x) &\equiv &\mathbf{e}_{2}^{\top }\gamma _{\tau }(x),\mathit{\
and} \\
\beta _{n,x,\tau }(Y_{i},z) &\equiv &-\tilde{l}_{\tau }\left( Y_{i}-\gamma
_{\tau }^{\top }(x)\cdot H\cdot c\left( z\right) \right) \mathbf{e}%
_{2}^{\top }M_{n,\tau }^{-1}(x)c\left( z\right) K\left( z\right) \text{.}
\end{eqnarray*}%
\newline
(ii) \textit{Suppose that the conditions of Theorem 3(ii) hold. Then
Asumptions A1-A3, A5-A6, and B1-B4 in LSW hold with the following
definitions:} $J=1,\ r_{n}\equiv \sqrt{nh^{3}},$%
\begin{eqnarray*}
v_{n,\tau }(x) &\equiv &\mathbf{e}_{2}^{\top }\{\gamma _{\tau
_{1}}(x)-\gamma _{\tau _{2}}(x)\},\ \mathit{and} \\
\beta _{n,x,\tau }(Y_{i},z) &\equiv &\alpha _{n,x,\tau _{1}}(Y_{i},z)-\alpha
_{n,x,\tau _{2}}(Y_{i},z)\text{,}
\end{eqnarray*}%
\textit{where} \textit{the set }$\mathcal{T}$\textit{\ in LSW} \textit{is
replaced by }$\mathcal{T}\times \mathcal{T}$\textit{\ here, and}%
\begin{equation*}
\alpha _{n,x,\tau }(Y_{i},z)\equiv -\tilde{l}_{\tau }\left( Y_{i}-\gamma
_{\tau }^{\top }(x)\cdot H\cdot c\left( z\right) \right) \mathbf{e}%
_{2}^{\top }M_{n,\tau }^{-1}(x)c\left( z\right) K\left( z\right) .
\end{equation*}
\end{LemmaMIQ}

\begin{proof}[Proof of Lemma MIQ]
(i) First, Assumption A1 in LSW follows from Theorem 1, with the error rate
in the asymptotic linear representation fulfills the rate $o_{P}(h^{1/2})$
 by the condition: $r>3/2$. Assumption A2 follows
because $\beta _{n,x,\tau }(Y_{i},z)$ has a multiplicative component of $%
K(z) $ having a compact support. As for Assumption A3, we can  use
Lemma 2 in LSW in combinations of Lipschitz continuity of $f_{\cdot }(\cdot
|\cdot )$ and $\gamma _{\cdot }(\cdot )$ to verify the assumption. 
As we take $\hat{\sigma}_{\tau ,j}(x)=\hat{\sigma%
}_{\tau ,j}^{\ast }(x)=1$, Assumptions A5 and B3 are trivially satisfied
with the choice of $\sigma _{n,\tau ,j}(x)=1$.$\ $Assumption A6(i) is
satisfied because $\beta _{n,x,\tau ,j}$ is bounded. Assumptions Assumption
B1 follows by Lemma QR2, and Assumption B2 by Lemma 2 in LSW. Assumption B4
follows by Assumption MON2(ii).
(ii) The proof is similar and  details are omitted.
\end{proof}

\begin{proof}[Proof of Theorem 3]
The results follow from Theorem 1 from LSW combined with Lemma MIQ1.
Details are omitted.
\end{proof}

\bibliographystyle{econometrica}
\bibliography{LSW_1Aug2015}

\clearpage

%{\tiny
\begin{table}[tbph]
\caption{Results of Monte Carlo experiments}
\label{table1}
\begin{center}
%{\tiny
\begin{tabular}{ccccccccc}
\hline\hline
& \multicolumn{3}{c}{Null model} & \multicolumn{3}{c}{Alternative Model 1} & 
&  \\ 
Bandwidth & \multicolumn{3}{c}{Nominal level} & \multicolumn{3}{c}{Nominal
level} &  &  \\ 
($h$) & 0.10 & 0.05 & 0.01 & 0.10 & 0.05 & 0.01 &  &  \\ \hline
&  &  &  &  &  &  &  &  \\ 
0.9 & 0.111 & 0.057 & 0.020 & 0.995 & 0.975 & 0.780 &  &  \\ 
1.0 & 0.100 & 0.048 & 0.007 & 0.980 & 0.920 & 0.660 &  &  \\ 
1.1 & 0.077 & 0.036 & 0.005 & 0.905 & 0.755 & 0.375 &  &  \\ \hline
& \multicolumn{3}{c}{Alternative Model 2} & \multicolumn{3}{c}{Alternative
Model 3} &  &  \\ 
Bandwidth & \multicolumn{3}{c}{Nominal level} & \multicolumn{3}{c}{Nominal
level} &  &  \\ 
($h$) & 0.10 & 0.05 & 0.01 & 0.10 & 0.05 & 0.01 &  &  \\ \hline
&  &  &  &  &  &  &  &  \\ 
0.9 & 0.985 & 0.965 & 0.800 & 0.990 & 0.970 & 0.660 &  &  \\ 
1.0 & 1.000 & 0.995 & 0.935 & 1.000 & 0.990 & 0.820 &  &  \\ 
1.1 & 1.000 & 1.000 & 0.985 & 1.000 & 0.990 & 0.835 &  &  \\ \hline
& \multicolumn{3}{c}{Alternative Model 4} & \multicolumn{3}{c}{Alternative
Model 5} &  &  \\ 
Bandwidth & \multicolumn{3}{c}{Nominal level} & \multicolumn{3}{c}{Nominal
level} &  &  \\ 
($h$) & 0.10 & 0.05 & 0.01 & 0.10 & 0.05 & 0.01 &  &  \\ \hline
&  &  &  &  &  &  &  &  \\ 
0.9 & 1.000 & 0.960 & 0.540 & 0.995 & 0.980 & 0.845 &  &  \\ 
1.0 & 0.885 & 0.645 & 0.175 & 0.995 & 0.995 & 0.985 &  &  \\ 
1.1 & 0.310 & 0.120 & 0.010 & 0.995 & 0.990 & 0.935 &  &  \\ \hline\hline
\end{tabular}
%}
\end{center}
\end{table}
%}

%\end{document}

\begin{figure}[htbp]
\caption{True Function and Simulated Data}
\label{figure1}
\begin{center}
\makebox{
\includegraphics[origin=bl,scale=.37,angle=90]{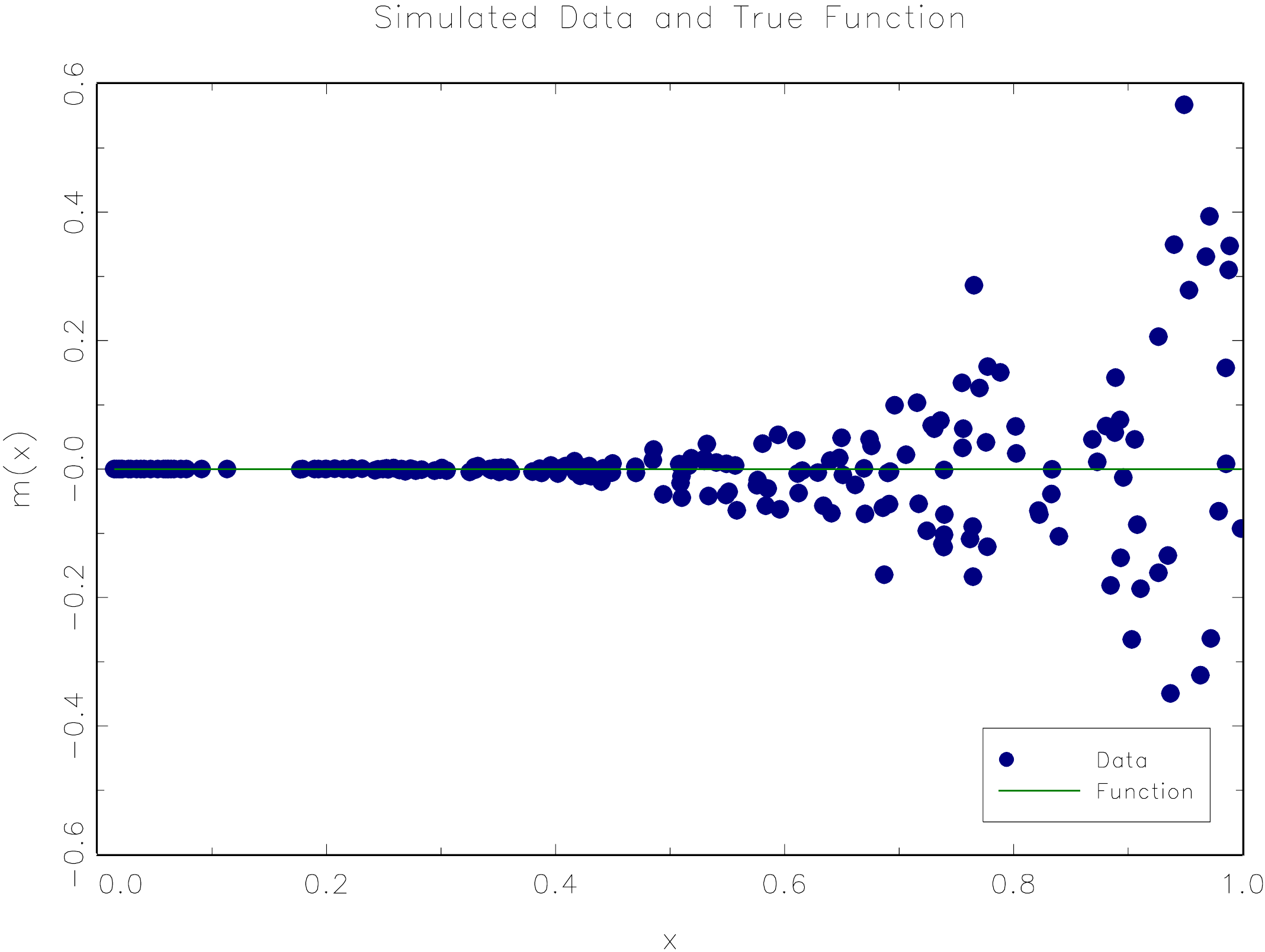}
}
\makebox{
\includegraphics[origin=bl,scale=.37,angle=90]{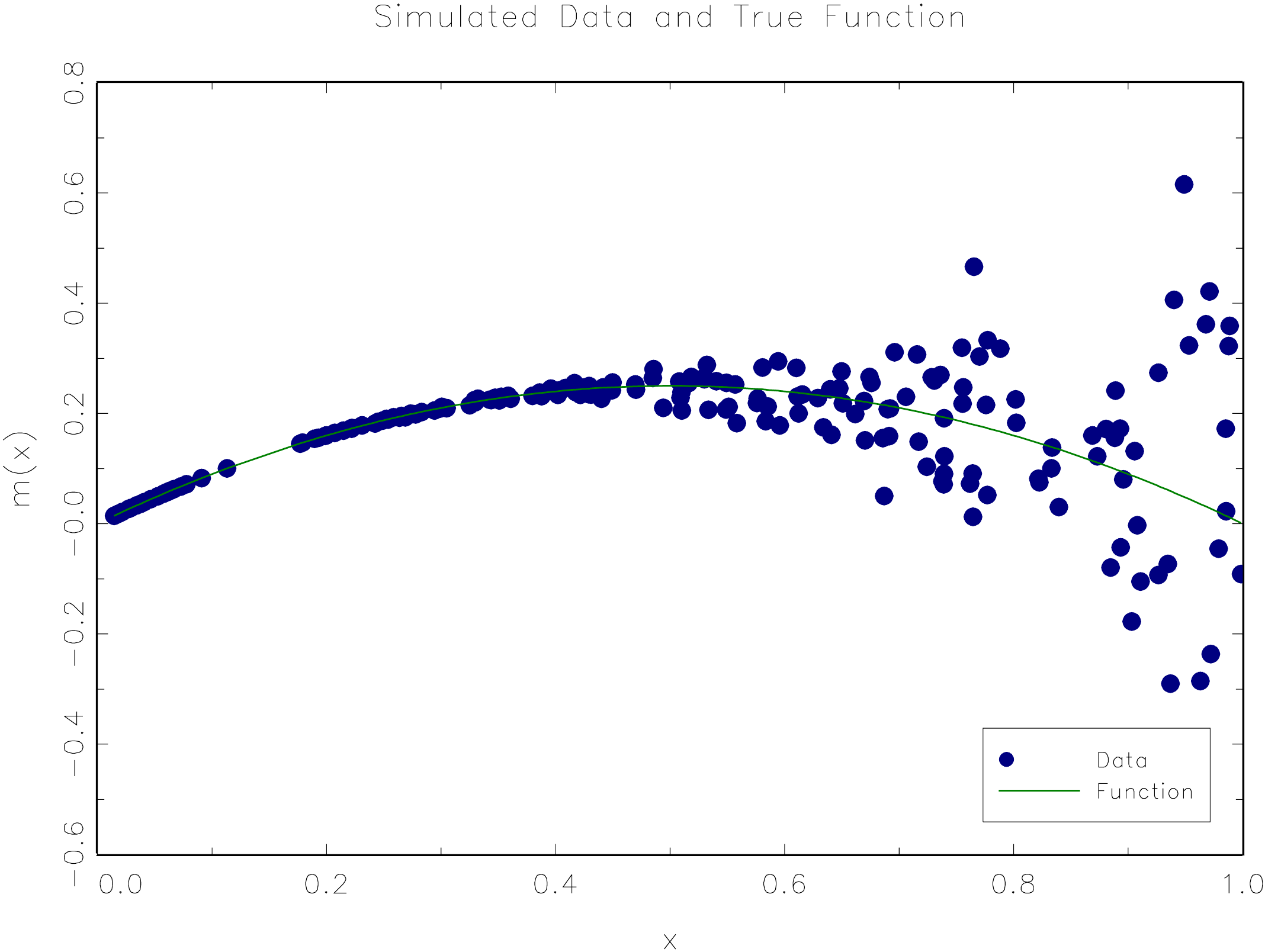}
}
\makebox{
\includegraphics[origin=bl,scale=.37,angle=90]{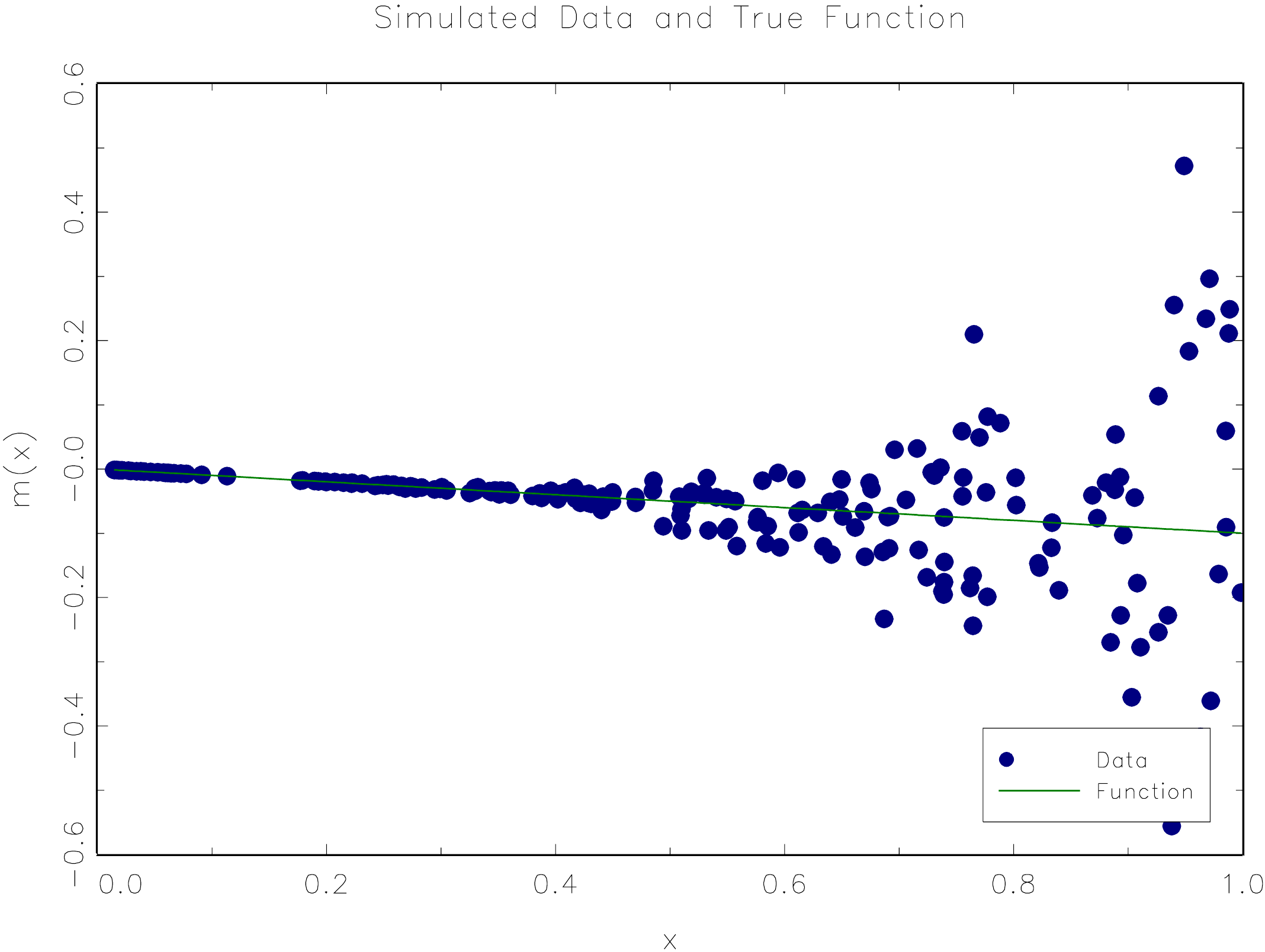}
}
\end{center}
\parbox{6in}{Note:  Each figure shows the true function and simulated data $\{(Y_i,X_i):1=1,\ldots,n=100 \}$ being generated from $Y_i = m_j(X_i) + U_i$, where
$X\sim \text{Unif}[0,1]$ and $U_i \sim X^4 \times \mathbf{N}(0,0.1^2)$, and $m_0(x) \equiv 0$, $m_1(x) = x(1-x)$, and $m_2(x) = -0.1x$, respectively.}
\end{figure}

\begin{figure}[htbp]
\caption{True Function and Simulated Data}
\label{figure2}
\begin{center}
\makebox{
\includegraphics[origin=bl,scale=.37,angle=90]{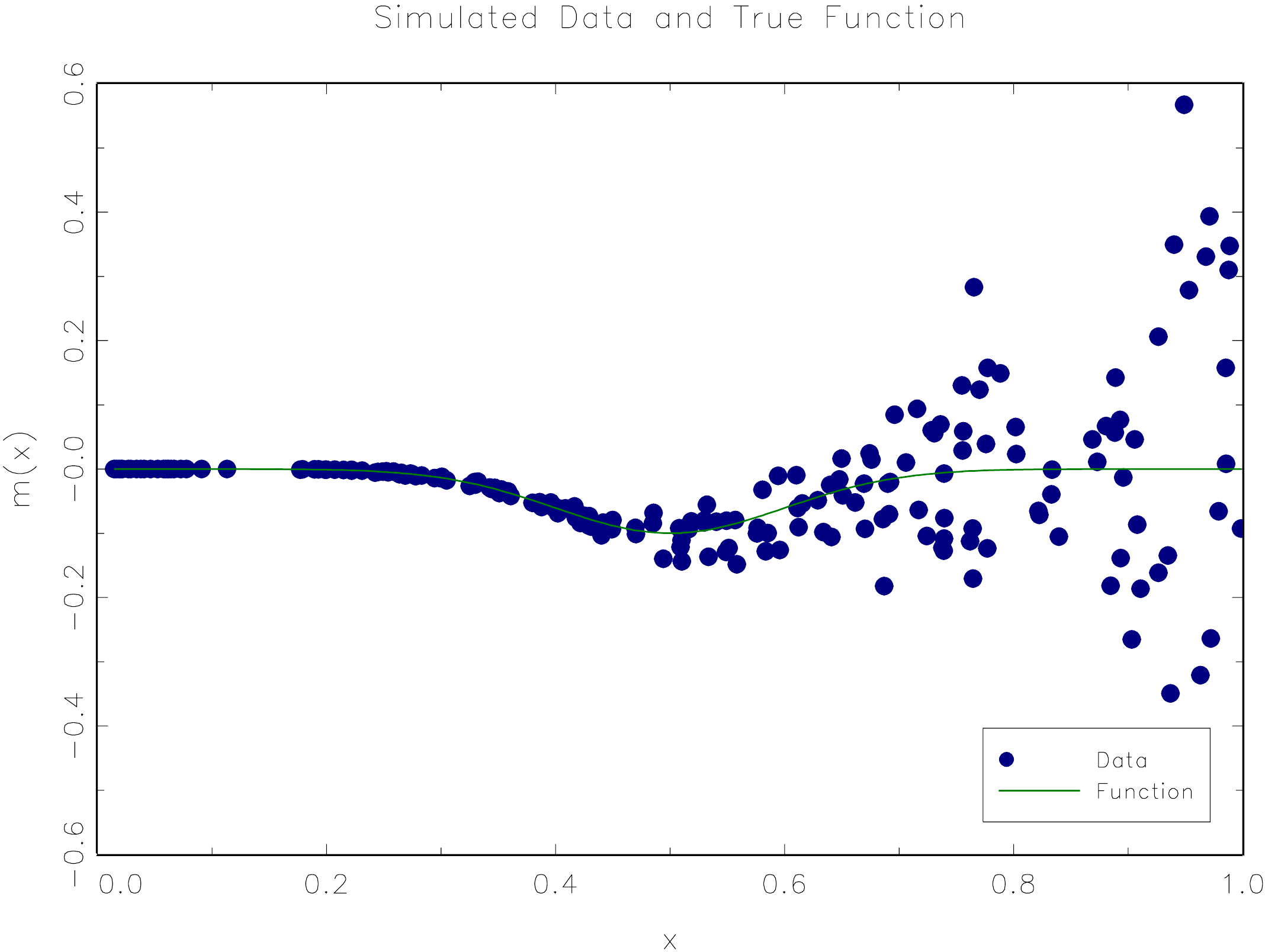}
}
\makebox{
\includegraphics[origin=bl,scale=.37,angle=90]{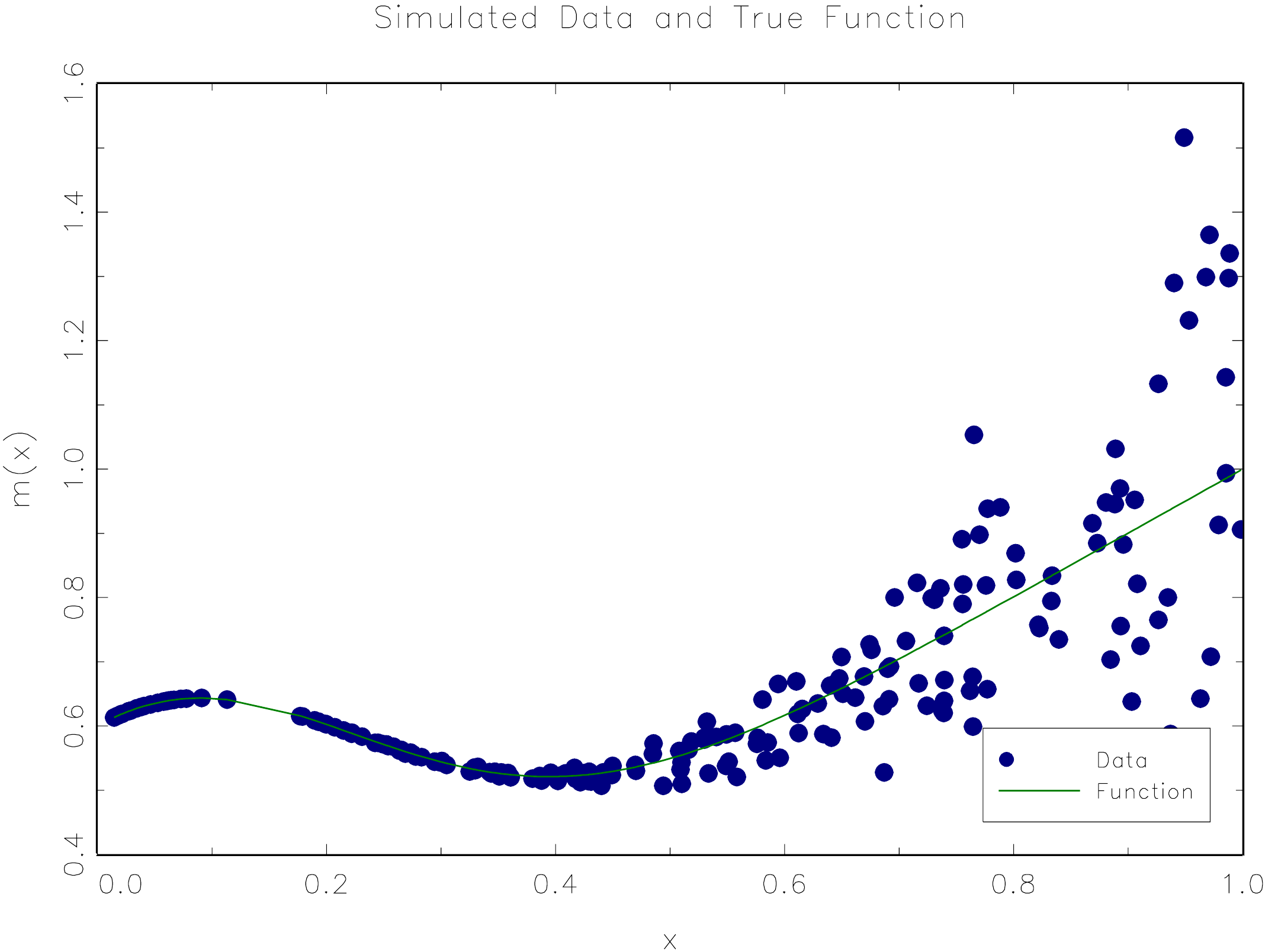}
}
\makebox{
\includegraphics[origin=bl,scale=.37,angle=90]{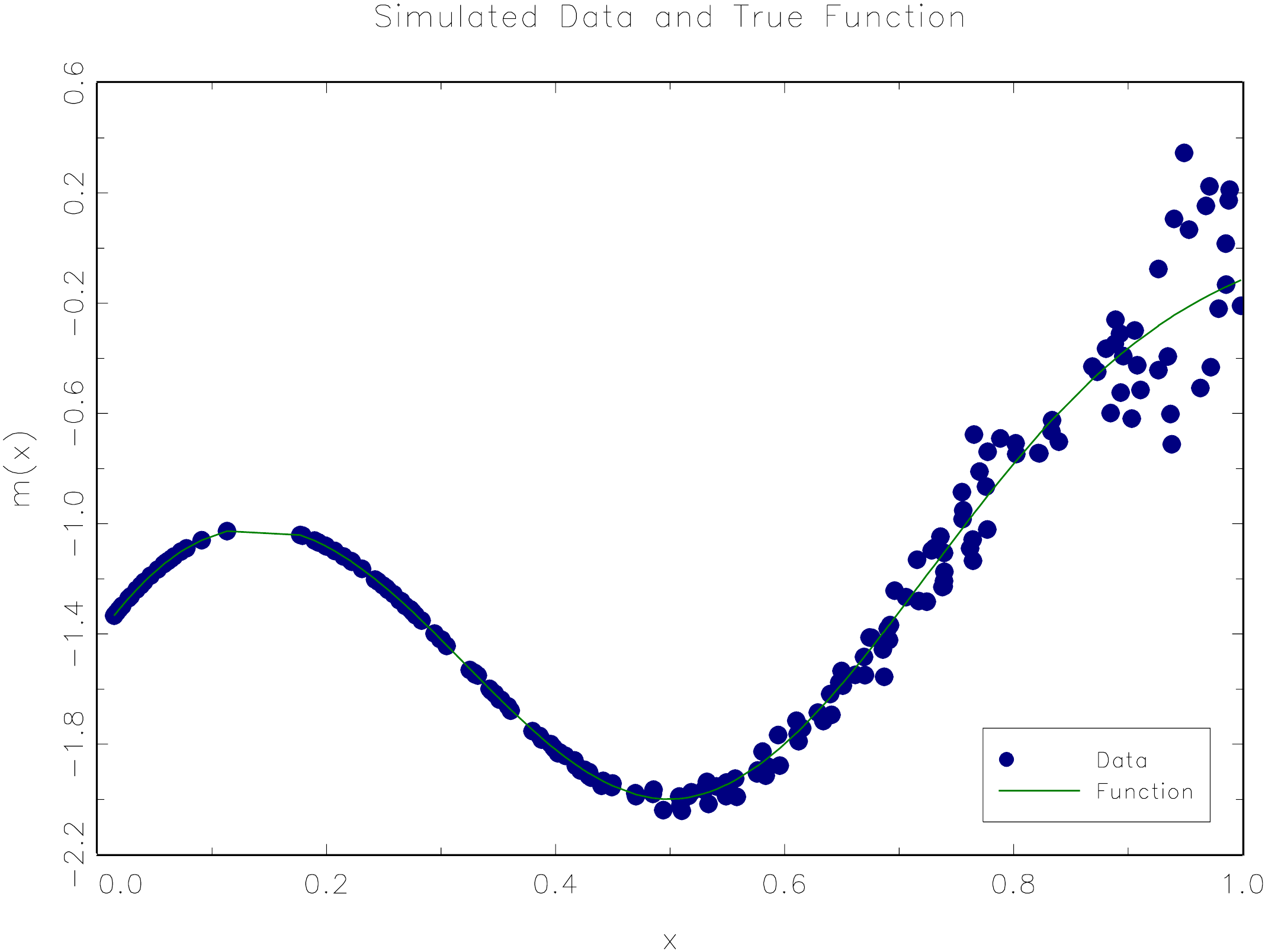}
}
\end{center}
\parbox{6in}{Note:  Each figure shows the true function and simulated data $\{(Y_i,X_i):1=1,\ldots,n=100 \}$ being generated from $Y_i = m_j(X_i) + U_i$, where
$X\sim \text{Unif}[0,1]$ and $U_i \sim X^4 \times \mathbf{N}(0,0.1^2)$, and $m_3(x) = -0.1 \exp(-50(x-0.5)^2)$, $m_4(x) = x + 0.6\exp(-10x^2)$, and $m_5(x) = [10((x-0.5)^3) - 2\exp(-10((x-0.5)^2))]1( x < 0.5) 
          + [0.1(x-0.5) - 2\exp(-10((x-0.5)^2))]1(x \geq 0.5)$, respectively.}
\end{figure}

\end{document}